\documentclass[12pt,paper=a4, headsepline,headings=normal,abstract=true]{scrartcl}

\usepackage[english]{babel}
\usepackage[utf8]{inputenc}
\usepackage[T1]{fontenc} 
\usepackage{textcomp}
\usepackage[margin=2cm]{geometry}
\usepackage[reqno]{amsmath}  
\usepackage{amssymb}          
\usepackage{amsthm}          
\usepackage{amstext}         
\usepackage{enumerate}
\usepackage{enumitem}	  
\usepackage{bbm}
\usepackage{color}
\usepackage{pifont}           
\usepackage{framed}
\usepackage{scrpage}
\usepackage{graphicx}
\usepackage{verbatim}
\usepackage{mathrsfs}
\usepackage{url}
\urldef{\emailvosshall}{\url}{vosshall@mathematik.uni-kl.de}

\setkomafont{pageheadfoot}{%
\normalfont\normalcolor\itshape\small
}
\setkomafont{pagenumber}{\normalfont\bfseries}

\makeatletter\@addtoreset{equation}{section}\makeatother

\allowdisplaybreaks

\theoremstyle{plain}      \newtheorem{theorem}{Theorem}[section]
                          \newtheorem{corollary}[theorem]{Corollary}
                          \newtheorem{proposition}[theorem]{Proposition}
													\newtheorem{condition}[theorem]{Condition}

\theoremstyle{remark}     \newtheorem{remark}[theorem]{Remark}
                          \newtheorem{lemma}[theorem]{Lemma}
													\newtheorem{example}[theorem]{Example}

\theoremstyle{definition} \newtheorem{definition}[theorem]{Definition}

\begin{document} 

\newcommand{\grad}{\nabla}
\newcommand{\D}{\partial}
\newcommand{\E}{\mathcal{E}}
\newcommand{\N}{\mathbb{N}}
\newcommand{\R}{\mathbb{R}_{\scriptscriptstyle{\ge 0}}}
\newcommand{\dom}{\mathcal{D}}
\newcommand{\ess}{\operatorname{ess~inf}}
\newcommand{\cem}{\operatorname{\text{\ding{61}}}}
\newcommand{\supp}{\operatorname{\text{supp}}}
\newcommand{\ca}{\operatorname{\text{cap}}}

\setenumerate[1]{label=(\roman*)}       
\setenumerate[2]{label=(\alph*)}

\begin{titlepage}
\title{\Large Interacting particle systems with sticky boundary}
\author{
\normalsize \sc Robert Vo\ss hall \footnote{University of Kaiserslautern, P.O.Box 3049, 67653
Kaiserslautern, Germany.}~\thanks{\emailvosshall}}
\date{\small \today}
\end{titlepage}
\maketitle

\pagestyle{headings}

\begin{abstract}
\noindent In this paper, we construct under general assumptions the stochastic dynamics of an interacting particle system in a bounded domain $\Omega$ with sticky boundary. Under appropriate conditions on the interaction the constructed process solves the underlying SDE for every starting point in the state space. Moreover, we also obtain a solution for q.e. starting point in the case of singular interactions which generalizes former results. Finally, the setting is applied to the case of particles diffusing in a chromatography tube.\\

\\
\end{abstract}

\section{Introduction}

In \cite{FGV14} a sticky reflected distorted Brownian motion on $[0,\infty)^n$, $n \in \mathbb{N}$, is constructed via Dirichlet forms and applied to stochastic interface models. Afterwards, the connection to random time changes and Girsanov transformations is investigated in \cite{GV14a}. In particular, strong Feller properties of the transition semigroup of a process associated to the underlying Dirichlet form are proven such that the existence result for weak solution of the underlying SDE of \cite{GV14a} is improved under appropiate assumptions on the drift. Moreover, the Dirichlet form construction of sticky reflected distorted Brownian motion on the half-line $[0,\infty)$ is generalized to general bounded domains $\Omega \subset \mathbb{R}^d$, $d \in \mathbb{N}$, in \cite{GV14b} such that the setting even allows a diffusion on $\partial \Omega$. In the present paper, we construct and analyze diffusions on $\overline{\Omega}^N$, $N \in \mathbb{N}$, with sojourn on the boundary. This type of diffusion desribes naturally a system of interacting particles with sticky boundary. In the independent case, i.e., the case without interaction, the setting reduces to $N$ independent diffusions, where each diffusion is of the type considered in \cite{GV14a}. We define the corresponding Dirichlet form and present the connections to random time changes and Girsanov transformations. Moreover, we calculate the corresponding $L^2$-generator for smooth functions and establish in this way the connection to the underlying martingale problem and SDE.\\
The construction allows very weak assumptions on the interaction and moreover, illustrates some effects which do not appear in the case of interacting particle systems with absorbing or reflecting boundary conditions. For example, in the case $N=1$ without drift the invariant measure for the sticky reflected Brownian motion on $\overline{\Omega}$ is given by $\lambda + \sigma$, where $\lambda$ denotes the Lebesgue measure on $\Omega$ and $\sigma$ the surface measure on $\partial \Omega$. Then, the invariant measure for general $N \in \mathbb{N}$ with additional drift is given by
\[ \mu=\varrho \prod_{i=1}^N (\lambda_i + \sigma_i), \]
where $\varrho$ is a suitable density. In the case of absorbing or reflecting boundary conditions the invariant measure can be derived similarly, but the surface measure does not appear. Hence, the structure of the product measure is much simpler. In this case, one obtains the Lebesgue measure on $\tilde{\Omega}:=\Omega^N$ and previous results apply. The only problem is that usually the boundary regularity decreases, since $\tilde{\Omega}$ possesses corners. In the case of the sticky boundary condition it is not possible anymore to reduce the setting for general $N \in \mathbb{N}$ to the case $N=1$. Therefore, it is necessary to analyze the structure of the problem in detail. \\

The investigated system of SDEs is of the form 
\begin{align*}
d\mathbf{X}^i_t =& \mathbbm{1}_{\Omega} (\mathbf{X}^i_t) \Big( dB^i_t + \frac{1}{2} \big( \frac{\nabla_i  \alpha_i}{\alpha_i} (\mathbf{X}^i_t) + \frac{\nabla_i \phi}{\phi} (\mathbf{X}_t) \big) dt \Big) - \mathbbm{1}_{\Gamma}(\mathbf{X}^i_t) \frac{\alpha_i}{\beta_i}(\mathbf{X}^i_t) ~n(\mathbf{X}^i_t) dt \\
 + & \delta~\mathbbm{1}_{\Gamma}(\mathbf{X}^i_t) \Big( dB_t^{\Gamma,i} + \big( \frac{\nabla_{\Gamma,i} \beta_i}{\beta_i} (\mathbf{X}_t^i)+ \frac{\nabla_{\Gamma,i} \phi}{\phi}(\mathbf{X}_t) \big) dt \Big), \quad i=1,\dots,N  \\
dB_t^{\Gamma,i}&=P(\mathbf{X}_t^i) \circ dB_t^i  \\ 
\mathbf{X}_0 =& x \in \overline{\Omega}^N,
\end{align*}
where $\delta \in \{0,1\}$, $(B_t)_{t \geq 0}$, $B_t=(B_t^1,\dots,B_t^N)$, is an $Nd$-dimensional standard Brownian motion, $n$ is the outward normal vector and $P$ is the projection on the tangent space. The particle interaction is given by $\nabla_i \ln \phi$ and $\nabla_{\Gamma,i} \ln \phi$, $i=1,\dots,N$, where $\nabla_{\Gamma}$ denotes the surface gradient. The precise definitions of $\nabla_i$ and $\nabla_{\Gamma,i}$ as well as $n$ and $P$ are given in Section \ref{generalnot}. The densities $\alpha_i$ and $\beta_i$, $i=1,\dots,N$, are only assumed to be continuous and to fulfill a weak differentiability condition whereas $\phi$ is $C^1$ (see also Condition \ref{conddiff} and Theorem \ref{thmsolSDE}). Note that the drift is nevertheless not necessarily Lipschitz continuous, since the densities are allowed to vanish on a set of measure zero. A similar system of SDEs has been investigated in \cite{Gra88} and applied to a model for molecules diffusing in a chromatography tube. We also consider such kind of applications and extend previous results to the case of singular interactions.\\

Our paper is organized as follows: In Section \ref{sectprel} basic notations are explained and some previous results are stated. In Section \ref{sectprocess} the underlying Dirichlet form is constructed and afterwards, in Section \ref{sectanapro} the associated diffusion is analyzed and the relations to random time changes and Girsanov transformations are presented. Finally, we apply the results in Section \ref{sectappl}.

\section{Preliminaries} \label{sectprel}

\subsection{General notation} \label{generalnot}

Throughout this paper, $\Omega \subset \mathbb{R}^d$, $d \geq 1$, denotes a nonempty bounded domain such that its boundary $\Gamma:=\partial \Omega$ is of Lebesgue measure zero. In the case $\delta=1$ we assume that $d \geq 2$. The standard scalar product in $\mathbb{R}^n$, $n \in \mathbb{N}$, is given by $(\cdot,\cdot)$ and norms in $\mathbb{R}^n$ by $|\cdot|$ (in particular, for the modulus in $\mathbb{R}$; eventually labeled by a lower index in order to distinguish norms). Similarly, $\Vert \cdot \Vert$ denotes norms in function spaces. The metric on $\mathbb{R}^d$ induced by the euclidean metric is denoted by $d_{\text{euc}}$. \\

For a vector $x \in \overline{\Omega}^N$, $N \in \mathbb{N}$, we use the representation $x=(x^1,\dots,x^N)$, where $x^i \in \overline{\Omega}$, $i=1,\dots,N$, is represented in the form $x^i=(x^i_1,\dots,x^i_d)$. We denote by $\nabla$ the gradient of a smooth function and by $\partial_{x^i_k}$, $i=1,\dots,N$, $k=1,\dots,d$, its partial derivatives. In the case $N=1$ we simply write $\partial_k$ for $k=1,\dots,d$. By $\nabla_i$, $i=1,\dots,N$, we denote the $d$-dimensional vector given by the partial derivatives with respect to the coordinates $x^i_k$, $k=1,\dots,d$. Moreover, $\nabla^2$ denotes the Hessian for functions mapping from subsets of $\mathbb{R}^d$ to $\mathbb{R}$ and $\Delta= \text{Tr}( \nabla^2 )$ the Laplacian. $\nabla^2_i$ and $\Delta_i$ are defined analogously. In the case of Sobolev functions we use the same notations in the weak sense.

\subsection{Submanifolds in the euclidean space}

In the following, the boundary $\Gamma$ of $\Omega$ is said to be Lipschitz continuous (respecktively $C^k$-smooth) if Definition 2.1 of \cite{GV14b} holds, i.e., $\Gamma$ is Lipschitz continuous (respectively $C^k$-smooth) if $\Omega$ is locally below the graph of a Lipschitz continuous (respectively $C^k$-) function and the graph coincides with $\Gamma$. In this case, we also simply say that $\Gamma$ is Lipschitz (respectively $C^k$) or that $\Omega$ has Lipschitz boundary (respectively $C^k$-boundary). Moreover, the surface measure on $\Gamma$ is denoted by $\sigma$ and the (outward) normal vector at a point $x \in \Gamma$ is denoted by $n(x)$ (supposed the boundary is smooth at $x$).

\begin{remark}
The definition of $n$ can be extended to a neighborhood of $x$ and $n$ is differentiable near $x$ if $\Gamma$ is $C^2$. 
\end{remark}

\begin{definition} \label{defproj}
Let $x \in \Gamma$ be such that the outward normal $n(x)$ exists. Define
\[ P(x):= E- n(x)n(x)^t \in \mathbb{R}^{d \times d}, \]
where $E$ is the $d \times d$ identity matrix. We call $P(x)$ the \textbf{orthogonal projection on the tangent space} at $x$. Note that $P(x)z=z- (n(x),z)~ n(x)$ for $z \in \mathbb{R}^d$.
\end{definition}

\begin{definition}
Let $f \in C^1(\overline{\Omega})$ and $x \in \Gamma$. Then we define (whenever $\Gamma$ is sufficiently smooth at $x$) the \textbf{gradient} of $f$ at $x$ along $\Gamma$ by 
\[ \nabla_{\Gamma} f (x):= P(x) \nabla f(x) \]
and if $f \in C^2(\overline{\Omega})$ the \textbf{Laplace-Beltrami} of $f$ at $x$ by
\[ \Delta_{\Gamma} f(x) := \text{Tr} ( \nabla_{\Gamma}^2 f(x))= \text{div}_{\Gamma} \nabla_{\Gamma} f(x)= \text{Tr}(P(x) \nabla (P(x) \nabla f(x))), \]
where $\text{div}_{\Gamma} \Phi := \text{Tr}(P \nabla \Phi)$ for $\Phi=(\Phi_1,\dots,\Phi_d) \in C^1(\overline{\Omega};\mathbb{R}^d)$ with $\nabla \Phi=J \Phi=(\nabla \Phi_1|\dots| \nabla \Phi_d)$.
Analogously, we define higher derivatives of order $k \in \mathbb{N}$. In this way, let $C^k(\Gamma_0)$ be the space of continuously differentiable functions on $\Gamma_0$ obtained by restriction of $C^k(\overline{\Omega})$-functions, where $\Gamma_0$ is an open subset of $\Gamma$ in the subspace topology. As usual, set $C^{\infty}(\Gamma_0):=\cap_{k \in \mathbb{N}}~ C^k(\Gamma_0)$. Moreover, in the case that $n$ is differentiable at $x$ we define the \textbf{mean curvature} of $\Gamma$ at $x$ by 
\[ \kappa(x):= \text{div}_{\Gamma}~ n (x). \]
\end{definition}

\begin{remark}
%For $f \in C^k(\overline{\Omega})$, $k \in \mathbb{N}$, the above definitions are in accordance with the ordinary definition on Riemannian manifolds for $f|_{\Gamma}$ using the inclusion $T_x \Gamma \hookrightarrow \mathbb{R}^d$, where $T_x \Gamma$ denotes the tangent space at $x \in \Gamma$. Conversely, if a function $f$, defined only on the Riemannian manifold $\Gamma$, is in $C^k(\Gamma)$, $k \in \mathbb{N}$, in the sense of manifolds, it is possible to extend the definition of $f$ to a $C^k$-function on an open set in $\mathbb{R}^d$ which contains $\Gamma$ and then it is feasible to use the definitions given above. Thus, $C^k(\Gamma)$, $k \in \mathbb{N}$, contains exactly the functions on $\Gamma$ obtained by restricting functions from $C^k(\overline{\Omega})$ to $\Gamma$.
For smooth functions, we have the divergence theorem
\begin{align} \label{divergence} \int_{\Gamma} (\Phi , \nabla_{\Gamma} g)~ d\sigma = - \int_{\Gamma} \text{div}_{\Gamma} \Phi~ g~d\sigma,
\end{align}
where $\Phi$ is $\mathbb{R}^d$-valued (see e.g. \cite[Chap. 2, Proposition 2.2]{Tay11}).
\end{remark}

The following lemma follows easily by calculation:

\begin{lemma} \label{lemcurv}
Assume that $\Gamma$ is $C^2$-smooth. Then
\[ ( P \nabla )^t P = - \kappa n. \] 
\end{lemma}

\begin{definition}
Let $\Gamma_0$ be an open subset of $\Gamma$ in the subspace topology. The \textbf{Sobolev space} $H^{1,k}(\Gamma_0)$, $k \geq 1$, is defined by $\overline{C^1(\Gamma_0)}^{\Vert \cdot \Vert_{H^{1,k}(\Gamma_0)}} \subset L^k(\Gamma_0;\sigma)$, i.e., the closure $C^1(\Gamma_0)$ with respect to the norm
\[ \Vert \cdot \Vert_{H^{1,k}(\Gamma_0)} := \big( \Vert \cdot \Vert_{L^k(\Gamma_0;\sigma)}^k + \Vert \nabla_{\Gamma} \cdot \Vert_{L^k(\Gamma_0;\sigma)}^k \big)^{\frac{1}{k}}. \]
\end{definition}

\begin{remark}
$H^{1,k}(\Gamma_0)$ can also be charaterized as the space of functions which are in local coordinates in the corresponding Sobolev space.\\
 If $f \in H^{1,k}(\Gamma_0)$ and $(f_n)_{n \in \mathbb{N}}$ is an approximating sequence of smooth functions, Cauchy in $H^{1,k}(\Gamma_0)$, we call the $L^k(\Gamma_0;\sigma)$-limit of $(\nabla_{\Gamma} f_n)_{n \in \mathbb{N}}$ the weak gradient of $f$ and denote it by $\nabla_{\Gamma} f$. In the case $\Gamma_0=\Gamma$, (\ref{divergence}) transfers from $f_n$ to $f$ using a continuity argument provided that $\Phi \in L^{k^{\prime}}(\Gamma;\sigma)$ for $\frac{1}{k}+\frac{1}{k^{\prime}}=1$.
\end{remark}

\subsection{Brownian motion on manifolds}

We shortly recall some facts about Brownian motion on $\Gamma$. For details about stochastic analysis on manifolds, we refer to \cite{HT94}, \cite{Hsu02} and \cite{IW89}.\\

By definition, Brownian motion $(B_t^{\Gamma})_{t \geq 0}$ on $\Gamma$ is a $\Gamma$-valued stochastic process that is generated by $\frac{1}{2} \Delta_{\Gamma}$, in analogy to Brownian motion on $\mathbb{R}^d$, in the sense that $(B_t^{\Gamma})_{t \geq 0}$ solves the martingale problem for $(\frac{1}{2} \Delta_{\Gamma},C^{\infty}(\Gamma))$. We recall the following:

\begin{lemma} 
Let $\Gamma$ be $C^2$-smooth. Then a solution of the Stratonovich SDE
\[ d\mathbf{X}_t= P(\mathbf{X}_t) \circ dB_t, \ \ \mathbf{X}_0 \in \Gamma, \]
is a Brownian motion on $\Gamma$, where $(B_t)_{t \geq 0}$ is a Brownian motion in $\mathbb{R}^d$.
\end{lemma}

\begin{proof}
See \cite[Chap. 3, Sect. 2]{Hsu02}.
\end{proof}

\begin{remark}
Note that the dimension of the driving Brownian motion $(B_t)_{t \geq 0}$ is strictly larger than the dimension of the submanifold $\Gamma$ and hence, according to \cite{Hsu02} the driving Brownian motion contains some extra information beyond what is usually provided by a Brownian motion on $\Gamma$. Furthermore, a solution of the above SDE is naturally $\Gamma$-valued, since $P(x)z$ is tangential to $\Gamma$ at $x$ for every $x \in \Gamma$ and $z \in \mathbb{R}^d$. In our application, it is natural to construct a Brownian motion on $\Gamma$ by means of a $d$-dimensional Brownian motion, since a Brownian motion on $\mathbb{R}^d$ is involved anyway.
\end{remark}

\noindent We also need the following result:

\begin{lemma}[It{\^o}-Stratonovich transformation rule]
Consider the Stratonovich integral in $\mathbb{R}^d$ given by 
\[ S(\mathbf{X}_t) \circ dB_t, \]
where $B=(B_t)_{t \geq 0}$ is a $d$-dimensional Brownian motion and $S:\mathbb{R}^d \mapsto \mathbb{R}^{d \times d}$ is $C^1$-smooth and symmetric. Then the It{\^o} form reads
\begin{align} \label{itoform} S(\mathbf{X}_t) dB_t + \frac{1}{2} \big((S \nabla)^t S \big)(\mathbf{X}_t) dt. \end{align}
\end{lemma}

Note that in the case $S=P$ the (\ref{itoform}) can be represented in the form
\[ P(\mathbf{X}_t) dB_t - \frac{1}{2} \kappa n(\mathbf{X}_t) dt \]
in view of Lemma \ref{lemcurv}.

%Rogers and Williams II p.185

\subsection{Sticky reflected diffusions on $\overline{\Omega}$} \label{sec1d}

In the following we recall the main results of \cite{GV14b}.\\

Assume that $\Gamma=\partial \Omega$ is Lipschitz continuous. Moreover, assume $\alpha \in L^1(\Omega; \lambda)$, $\alpha >0$ $\lambda$-a.e., and $\beta \in L^1(\Gamma; \sigma)$, $\beta >0$ $\sigma$-a.e..\\
Define
\begin{align} \label{densityN=1} \varrho:= \mathbbm{1}_{\Omega} ~ \alpha + \mathbbm{1}_{\Gamma} ~ \beta 
\end{align}
as well as
\[ \mu:= \varrho ~ (\lambda + \sigma) = \alpha \lambda + \beta  \sigma. \]
Note that the condition $\alpha \in L^1(\Omega; \lambda)$, $\alpha >0$ $\lambda$-a.e., and $\beta \in L^1(\Gamma; \sigma)$, $\beta >0$ $\sigma$-a.e. is equivalent to $\varrho \in L^1(\overline{\Omega};\lambda + \sigma)$, $\varrho >0$ $(\lambda + \sigma)$-a.e..

Let the symmetric and positive definite bilinear form $(\mathcal{E},\mathcal{D})$ be given by
\begin{align} \label{defform} \mathcal{E}(f,g):= \frac{1}{2} \int_{\Omega} (\nabla f, \nabla g)~ \alpha d\lambda + \frac{\delta}{2} \int_{\Gamma} (\nabla_{\Gamma} f,\nabla_{\Gamma} g) ~\beta d\sigma \ \text{ for } f,g \in \mathcal{D}:=C^1(\overline{\Omega}), \end{align}
where $(\cdot,\cdot)$ denotes the euclidean scalar product in $\mathbb{R}^d$ and $\delta \in \{0,1\}$. In addition, let 
\[ \mathcal{E}_{\Omega}(f,g):=  \frac{1}{2} \int_{\Omega} (\nabla f, \nabla g)~ \alpha d\lambda \ \text{ for } f,g \in \mathcal{D}_{\Omega}:=C^1(\overline{\Omega}) \]
as well as
\[ \mathcal{E}_{\Gamma}(f,g):= \frac{1}{2} \int_{\Gamma} (\nabla_{\Gamma} f,\nabla_{\Gamma} g) ~\beta d\sigma \ \text{ for } f,g \in \mathcal{D}_{\Gamma}:=C^1(\Gamma). \]
Note that $e(\mathcal{D})=e(\mathcal{D}_{\Omega})= \mathcal{D}_{\Gamma}$, where $e: C^1(\overline{\Omega}) \rightarrow C^1(\Gamma)$ is defined by the restriction of functions to $\Gamma$. In this terms, for $f,g \in \mathcal{D}$ we get 
\[ \mathcal{E}(f,g)= \mathcal{E}_{\Omega}(f,g) + \delta ~\mathcal{E}_{\Gamma}(f,g).\]

In order to prove closability of $(\mathcal{E},\mathcal{D})$, we need an additional assumption on the density $\varrho$. Define
\[ R_{\alpha}(\Omega) :=\{ x \in \Omega : \int_{\{ y \in \Omega : |x-y| < \epsilon \}} \alpha^{-1} d\lambda < \infty \ \text{ for some } \epsilon >0 \} \]
and analogously $R_{\beta}(\Gamma)$ with $\Omega$ replaced by $\Gamma$ and $\lambda$ replaced by $\sigma$. We assume that $\alpha =0$ $\lambda$-a.e. on $\Omega \backslash R_{\alpha}(\Omega)$ and additionally $\beta=0$ $\sigma$-a.e. on $\Gamma \backslash R_{\beta}(\Gamma)$ if $\delta=1$(Hamza condition).

Under these assumptions the following holds true:

\begin{theorem} \label{thmN=1}
The symmetric and positive definite bilinear form $(\mathcal{E},D)$ is denesly defined and closable on $L^2(\overline{\Omega};\mu)$. Its closure $(\mathcal{E},D(\mathcal{E}))$ is a recurrent, strongly local, regular, symmetric Dirichlet form on $L^2(\overline{\Omega};\mu)$.
\end{theorem}

As an immediate consequence we obtain the following theorem:

\begin{theorem} \label{thmdiff1}
There exists a conservative diffusion process (i.e. a strong Markov process with continuous sample paths and infinite life time)
\[ \mathbf{M}:=\big( \mathbf{\Omega}, \mathcal{F}, (\mathcal{F}_t)_{t \geq 0}, (\mathbf{X}_t)_{t \geq 0}, (\Theta_t)_{t \geq 0}, (\mathbf{P}_x)_{x \in \overline{\Omega}} \big) \]
with state space $\overline{\Omega}$ which is properly associated with $(\mathcal{E},D(\mathcal{E}))$, i.e., for all ($\mu$-versions of) $f \in \mathcal{B}_b(\overline{\Omega}) \subset L^2(\overline{\Omega};\mu)$ and all $t >0$ the function
\[ \overline{\Omega} \ni x \mapsto p_t f(x):= \mathbb{E}_x \big(f(\mathbf{X}_t) \big) := \int_{\Omega} f(\mathbf{X}_t) d\mathbf{P}_x \in \mathbb{R} \]
is a quasi continuous version of $T_t f$. $\mathbf{M}$ is up to $\mu$-equivalence unique. In particular, $\mathbf{M}$ is $\mu$-symmetric, i.e.,
\[ \int_{\overline{\Omega}} p_t f ~ g ~d\mu = \int_{\overline{\Omega}} f ~ p_t g ~d\mu \ \text{ for all } f,g \in \mathcal{B}_b(\overline{\Omega}) \ \text{ and all } t >0, \]
and has $\mu$ as invariant measure, i.e.,
\[ \int_{\overline{\Omega}} p_t f~ d\mu = \int_{\overline{\Omega}} f~ d\mu \ \text{ for all } f \in \mathcal{B}_b(\overline{\Omega}) \ \text{ and all } t >0. \]
\end{theorem}

If we assume the stronger conditions that $\Gamma$ is $C^2$-smooth and $\alpha, \beta \in C(\overline{\Omega})$, $\alpha >0$ $\lambda$-a.e. on $\Omega$, $\beta >0$ $\sigma$-a.e. on $\Gamma$ such that $\sqrt{\alpha} \in H^{1,2}(\Omega)$ and additionally $\sqrt{\beta} \in H^{1,2}(\Gamma)$ if $\delta=1$, it is possible to determine the generator of $(\mathcal{E},D(\mathcal{E}))$ for functions in $C^2(\overline{\Omega})$. The explicit representation of the generator allows to analyze the dynamics of $\mathbf{M}$:

\begin{theorem} \label{thmsolSDE1}
$\mathbf{M}$ is a solution to the SDE
\begin{align}
d\mathbf{X}_t =& \mathbbm{1}_{\Omega} (\mathbf{X}_t) \Big( dB_t + \frac{1}{2} \frac{\nabla \alpha}{\alpha} (\mathbf{X}_t) dt \Big) - \mathbbm{1}_{\Gamma}(\mathbf{X}_t) \frac{\alpha}{\beta}(\mathbf{X}_t) ~n(\mathbf{X}_t) dt   \notag \\ 
&+ \delta ~ \mathbbm{1}_{\Gamma}(\mathbf{X}_t) \Big( dB_t^{\Gamma} + \frac{1}{2} \frac{\nabla_{\Gamma} \beta}{\beta} (\mathbf{X}_t) dt \Big), \label{SDEN=1} \\
dB_t^{\Gamma} =& P(\mathbf{X}_t) \circ dB_t, \notag \\
\mathbf{X}_0 =& x, \notag
\end{align}
for quasi every starting point $x \in \overline{\Omega}$, where $(B_t)_{t \geq 0}$ is a $d$-dimensional standard Brownian motion, i.e.,
\begin{align}
\mathbf{X}_t = x &+ \int_0^t \mathbbm{1}_{\Omega}(\mathbf{X}_s) dB_s + \int_0^t \mathbbm{1}_{\Omega}(\mathbf{X}_s) \frac{1}{2} \frac{\nabla \alpha}{\alpha}(\mathbf{X}_s) ds \notag \\
&+ \delta \int_0^t \mathbbm{1}_{\Gamma}(\mathbf{X}_s) P(\mathbf{X}_s) dB_s - \delta \int_0^t  \mathbbm{1}_{\Gamma}(\mathbf{X}_s) \frac{1}{2} \kappa(\mathbf{X}_s) n(\mathbf{X}_s) ds \label{qesolution} \\
&+ \delta \int_0^t  \mathbbm{1}_{\Gamma}(\mathbf{X}_s) \frac{1}{2} \frac{\nabla_{\Gamma} \beta}{\beta}(\mathbf{X}_s) ds - \int_0^t \frac{1}{2} \frac{\alpha}{\beta}(\mathbf{X}_s)  \mathbbm{1}_{\Gamma}(\mathbf{X}_s)n(\mathbf{X}_s) ds \notag
\end{align}
almost surely under $\mathbf{P}_x$ for quasi every $x \in \overline{\Omega}$.
\end{theorem}

If we suppose addtionally to the assumptions of Theorem \ref{thmsolSDE1} that there exists $p \geq 2$ with $p > \frac{d}{2}$ such that
\[ \mathbbm{1}_{\Omega} \frac{ | \nabla \alpha | }{\alpha} + \delta ~\mathbbm{1}_{\Gamma} \frac{| \nabla \beta |}{\beta} \in L^p_{\text{loc}}(\overline{\Omega} \cap \{ \varrho >0\}; \mu) \]
and $\text{cap}_{\mathcal{E}}(\{ \varrho=0\})=0$,
we even obtain a stronger version:

\begin{theorem} There exists a conservative diffusion process
\[ \mathbf{M}=\big( \mathbf{\Omega}, \mathcal{F}, (\mathcal{F}_t)_{t \geq 0}, (\mathbf{X}_t)_{t \geq 0}, (\Theta_t)_{t \geq 0}, (\mathbf{P}_x)_{x \in \overline{\Omega} \cap \{ \varrho >0\}} \big)  \]
with state space $\overline{\Omega} \cap \{ \varrho >0\}$ such that $\mathbf{M}$ solves (\ref{SDEN=1}) for every $x \in \overline{\Omega} \cap \{\varrho >0\}$. Moreover, its Dirichlet form is given by $(\mathcal{E},D(\mathcal{E}))$ on $L^2(\overline{\Omega} \cap \{ \varrho >0\};\mu)$ and the transition semigroup $(p_t)_{t >0}$ of
$\mathbf{M}$ is $\mathcal{L}^p$-strong Feller, i.e., $p_t (\mathcal{L}^p (\overline{\Omega} \cap \{\varrho >0\};\mu)) \subset C(\overline{\Omega} \cap \{ \varrho >0\})$. In particular, $(p_t)_{t >0}$ it strong Feller, i.e., $p_t (\mathcal{B}_b(\overline{\Omega} \cap \{\varrho >0\})) \subset C(\overline{\Omega} \cap \{ \varrho >0\})$. Furhtermore, $\mathbf{M}$ has a sticky boundary behavior, i.e.,
\begin{align*} \lim_{t \rightarrow \infty} \frac{1}{t} \int_0^t \mathbbm{1}_{\Gamma}(\mathbf{X}_s) ds  >0 
\end{align*} 
$\mathbf{P}_x$-a.s. for every $x \in \overline{\Omega} \cap \{ \varrho >0\}$ such that $x$ is in a component of $\overline{\Omega} \cap \{ \varrho >0\}$ intersecting $\Gamma$.
\end{theorem}

\section{The Dirichlet form and the associated Markov process}\label{sectprocess}

\subsection{General setting} \label{secgeneral}

Assume that $\Gamma:=\partial \Omega$ is Lipschitz continuous. Let $(\mathcal{G},D(\mathcal{G}))$ be the recurrent, strongly local, regular, symmetric Dirichlet form on $L^2(\overline{\Omega};\lambda + \sigma)$ in accordance with Theorem \ref{thmN=1} for $\alpha=\beta=\mathbbm{1}_{\overline{\Omega}}$.
Set $\Lambda:= \overline{\Omega}^N$. Note that $\Lambda \subset \mathbb{R}^{Nd}$ is connected and compact. In the following we use the product measure $\prod_{i=1}^N \mu_i$ on $\Lambda$, where $\mu_i:=\lambda_i + \sigma_i$ is defined on $\overline{\Omega}$ and the index $i$ gives reference to the corresponding coordinate. For functions $f,g \in C^1(\Lambda)$, $i \in I$ and $x^j \in \overline{\Omega}$ for $j \in I$, $j \neq i$, define 
{\small \[ \mathcal{E}^i(f,g)(x^1,\dots,x^{i-1},x^{i+1},\dots,x^N) := \mathcal{G}(f(x^1,\dots,x^{i-1},\cdot,x^{i+1},\dots,x^N),g(x^1,\dots,x^{i-1},\cdot,x^{i+1},\dots,x^N)). \]}

Define the symmetric bilinear form $(\tilde{\mathcal{E}},\mathcal{D})$ by
\begin{align} \tilde{\mathcal{E}}(f,g)&:= \sum_{i=1}^N \int_{\overline{\Omega}^{N-1}} \mathcal{E}^i(f,g) \prod_{j \neq i} d\mu_j \quad \text{for } f,g \in \mathcal{D}:=C^1(\Lambda) \label{formbouleau} .
\end{align}
Using the definition of the form $\mathcal{G}$ yields
\begin{align}
\tilde{\mathcal{E}}(f,g) &= \frac{1}{2} \int_{\Lambda} \underbrace{\sum_{i=1}^N \big( \mathbbm{1}_{\Lambda^{i,\Omega}}~ (\nabla_i f, \nabla_i g) + \delta ~\mathbbm{1}_{\Lambda^{i,\Gamma}} (\nabla_{\Gamma,i} f, \nabla_{\Gamma,i} g) \big)}_{=: \Gamma(f,g)} \prod_{j=1}^N d\mu_j, \label{formsquare}
\end{align}
where $\Lambda^{i,\Omega}:=\{ x=(x_1,\dots,x_N) \in \Lambda|~ x_i \in \Omega \}$ and $\Lambda^{i,\Gamma}:=\{ x=(x_1,\dots,x_N) \in \Lambda|~ x_i \in \Gamma \}$. In particular, $\Lambda^{i,\Omega} ~\dot{\cup}~ \Lambda^{i,\Gamma}=\Lambda$ for every $i=1,\dots,N$.

\begin{condition} \label{condL1} $\varrho \in L^1(\Lambda;~\prod_{i=1}^N \mu_i)$,  $\varrho >0$ $\prod_{i=1}^N \mu_i$-a.e..
\end{condition}

Define $\mu$ by $\mu:=\varrho~ \prod_{j=1}^N \mu_j$ and $(\mathcal{E},\mathcal{D})$ by
\begin{align} \mathcal{E}(f,g)&:= \frac{1}{2} \int_{\Lambda} \sum_{i=1}^N \big( \mathbbm{1}_{\Lambda^{i,\Omega}}~ (\nabla_i f, \nabla_i g) + \delta ~\mathbbm{1}_{\Lambda^{i,\Gamma}} (\nabla_{\Gamma,i} f, \nabla_{\Gamma,i} g) \big) ~\varrho \prod_{j=1}^N d\mu_j \label{Nform} \\
&=\frac{1}{2} \int_{\Lambda} \Gamma(f,g)~ \varrho \prod_{j=1}^N d\mu_j \notag \\
&=\frac{1}{2} \int_{\Lambda} \Gamma(f,g)~ d\mu \notag
\end{align}
Note that the case $\delta=0$ correpsonds to the setting of a system of particles which has a sticky but static boundary behavior. Then, the bilinear form $(\mathcal{E},\mathcal{D})$ can be written in the simpler form
\[ \mathcal{E}(f,g)=\frac{1}{2} \int_{\Lambda} \sum_{i=1}^N \mathbbm{1}_{\Lambda^{i,\Omega}} (\nabla_i f,\nabla_i g) ~\varrho \prod_{j=1}^N d\mu_j \quad \text{for } f,g \in \mathcal{D}.\]

By the fact that $\mu$ is a Baire measure on $\Lambda$ we get the following result:

\begin{proposition} \label{propdense}
Under Condition \ref{condL1} we have that $C^{\infty}(\Lambda)$ is dense in $L^2(\Lambda;\mu)$.
\end{proposition}

Define $\Lambda_B:=\{ x \in \Lambda|~x_i \in \Omega \text{ for } i \in B,~x_i \in \Gamma \text{ for } i \in I \backslash B \}$ and $\nu_B:=\prod_{i \in B} \lambda_i \prod_{i \in I \backslash B} \sigma_i$. Then
\[ \Lambda = \dot{\bigcup}_{B \subset I} ~\Lambda_B \quad \text{and } \mu=\sum_{B \subset I} \underbrace{\varrho~ \nu_B}_{=: \mu_B}. \]
In this terms it holds
\[ \mathcal{E}(f,g)=\sum_{\emptyset \neq B \subset I} \mathcal{E}_B(f,g) \quad \text{for  } f,g \in \mathcal{D}, \]
where
\[ \mathcal{E}_B(f,g):= \frac{1}{2} ~\int_{\Lambda_B} \sum_{i \in B} (\nabla_i f,\nabla_i g) + \delta \sum_{i \in I \backslash B} (\nabla_{\Gamma,i} f, \nabla_{\Gamma,i} g)~ d\mu_B. \]
Moreover, define for $x \in \Gamma^{N-|B|}$, $B \neq \emptyset$, and $\varrho \in L^1(\Lambda;\prod_{i=1}^N \mu_i)$
\[ R_{\varrho}^{\Omega}(B,x):=\{ y \in \Omega^{|B|}|~\int_{\{z \in \Omega^{|B|}|~|z-y|<\epsilon\}} \varrho^{-1} ~\prod_{i \in B} \lambda_i < \infty \text{ for some } \epsilon >0 \}. \]
The dependence of $x$ is given in the sense that the variables of $\varrho$ given by the index set $I \backslash B$ are fixed by the components of $x$. Since $\varrho$ is an element of $L^1(\Lambda_B;\mu_B)$, $R_{\varrho}^{\Omega}(B,x)$ is only defined $\prod_{i \in I \backslash B} \sigma_i$ almost everywhere. 
Similarly, for $y \in \Omega^{|B|}$ let $R_{\varrho}^{\Gamma}(B,y)$ be given by
\[ R_{\varrho}^{\Gamma}(B,y):=\{ x \in \Gamma^{N-|B|}|~\int_{\{z \in \Gamma^{N-|B|}|~|z-x|<\epsilon\}} \varrho^{-1} ~\prod_{i \in I \backslash B} \sigma_i < \infty \text{ for some } \epsilon >0 \}. \]
In this case, the variables of $\varrho$ given by the index set $B$ are fixed by the components of $y$ and $R_{\varrho}^{\Gamma}(B,y)$ is only defined $\prod_{i \in B} \lambda_i$ almost everywhere.
Note that in both cases $B$ determines the components which are \textit{not} at the boundary. \\
The following condition is a generalized version of the usual Hamza condition (see e.g. \cite[Chapter II, (2.4)]{MR92}):

\begin{condition}[Hamza condition] \label{condhamza}
It holds 
\begin{enumerate}
\item[(H1)] $\varrho=0$ $\prod_{i \in B} \lambda_i$-a.e. on $\Omega^{|B|} \backslash R_{\varrho}^{\Omega}(B,x)$ for $\prod_{i \in I \backslash B} \sigma_i$-a.e. $x \in \Gamma^{N-|B|}$ for every $\emptyset \neq B \subset I$
\end{enumerate}
and if $\delta=1$ additionally
\begin{enumerate}
\item[(H2)] $\varrho=0$ $\prod_{i \in I \backslash B} \sigma_i$-a.e. on $\Gamma^{N-|B|} \backslash R_{\varrho}^{\Gamma}(B,y)$ for $\prod_{i \in B} \lambda_i$-a.e. $y \in \Omega^{|B|}$ for every $ B \subsetneq I$. 
\end{enumerate}
(For $B=I$ the condition \textit{(H1)} and for $B=\emptyset$ the condition \textit{(H2)} reduce to the ordinary Hamza condition.)
\end{condition}

\begin{remark} \begin{enumerate}
\item Condition \ref{condhamza} is a natrual generalizaion of the ordinary Hamza condition, since in the present setting of sticky particles we are also interested in dynamics whenever one (or several) particles are located at the boundary. The set $B$ determines the components inside $\Omega$ and its complement $I \backslash B$ the components on $\Gamma$. Thus, \textit{(H1)} ensures that the Hamza condition for the components inside $\Omega$ is fulfilled, wherever the remaining components stick on $\Gamma$. Since we are also interested in dynamics on $\Gamma$ if $\delta=1$, \textit{(H2)} is the corresponding condition in this case.
\item For $\Omega=(0,\infty)$ \textit{(H1)} of Condition \ref{condhamza} coincides with \cite[Condition 2.7]{FGV14} (disregarding that $(0,\infty)$ is unbounded), since in this case the surface measure on $\Gamma$ reduces to the case of the point measure in $0$.
\end{enumerate}
\end{remark}

\begin{remark} \label{remcontHamza}
If $\varrho$ is e.g. continuous on $\Lambda$ and positive $\prod_{i=1}^N \mu_i$-a.e., then $\varrho$ is outside the set $\{\varrho=0\}$ locally bounded away from zero and hence, $R_{\varrho}^{\Omega}(B,x)=\{ y \in \Omega^{|B|}|~ \varrho(z_{(B,x,y)})>0\}$ and $R_{\varrho}^{\Gamma}(B,y)=\{ x \in \Gamma^{N-|B|}|~ \varrho(z_{(B,x,y)})>0\}$, where \[z^i_{(B,x,y)}=\left\{\begin{array}{cl} y^{\gamma_B(i)}, & \mbox{if } i \in B \\ 
                                      x^{i-\gamma_B(i)}, & \mbox{if } i \in I \backslash B \end{array}\right.
                                       \]
with $\gamma_B:I \rightarrow \{1,\dots,|B|\}, i \mapsto |\{ 1 \leq j \leq i|~j \in B\}|$. Hence, 
\[ \Omega^{|B|} \backslash R_{\varrho}^{\Omega}(B,x)=\{ y \in \Omega^{|B|}|~\varrho(z_{(B,x,y)})=0\} \]
and 
\[ \Gamma^{N-|B|} \backslash R_{\varrho}^{\Gamma}(B,x)=\{ y \in \Gamma^{N-|B|}|~\varrho(z_{(B,x,y)})=0\} \]
for every $x \in \Gamma^{N-|B|}$, $y \in \Omega^{|B|}$ and Condition \ref{condhamza} is fulfilled.
\end{remark}

\begin{lemma} \label{lemmaclosable}
Suppose that Condition \ref{condL1} and Condition \ref{condhamza} are satisfied. Then the bilinear form $(\mathcal{E},\mathcal{D})$ is closable on $L^2(\Lambda;\mu)$. 
\end{lemma}

\begin{proof}
Let $(f_k)_{k \in \mathbb{N}}$ be an $\mathcal{E}$-Cauchy sequence in $\mathcal{D}$ such that $f_k \rightarrow 0$ in $L^2(\Lambda;\mu)$ as $k \rightarrow \infty$. In particular, $(f_k)_{k \in \mathbb{N}}$ is $\mathcal{E}_B$-Cauchy and converges to $0$ in $L^2(\Lambda_B;\mu_B)$ for every $\emptyset \neq B \subset I$. Thus, by definition of $\mathcal{E}_B$ we have that $(\partial_j f_k)_{k \in \mathbb{N}}$ is Cauchy in $L^2(\Lambda_B; \mu_B)$ for every $j=d(i-1)+l$, where $i \in B$ and $l \in \{1,\dots,d\}$ and hence, $\partial_j f_k \rightarrow h_j \in L^2(\Lambda_B;\mu_B)$ as $k \rightarrow \infty$. In other words,
\begin{align} \label{integral1} \int_{\Gamma^{N-|B|}} \int_{\Omega^{|B|}} (\partial_j f_k -h_j)^2 \varrho \prod_{i \in B} \lambda_i \prod_{i \in I \backslash B} \sigma_i \rightarrow 0 \quad \text{as } k \rightarrow \infty.
\end{align}
Therefore, it exists a subsequence $(\partial_j f_{k_l} )_{l \in \mathbb{N}}$ such that $\partial_j f_{k_l} \rightarrow h_j$ as $l \rightarrow \infty$ in $L^2(\Omega^{|B|}; \varrho \prod_{i \in B} \lambda_i)$ $\prod_{i \in I \backslash B} \sigma_i$-a.e. and similarly, $f_{k_l} \rightarrow 0$ as $l \rightarrow \infty$ in $L^2(\Omega^{|B|}; \varrho \prod_{i \in B} \lambda_i)$ $\prod_{i \in I \backslash B} \sigma_i$-a.e.. This implies that $h_j=0$ on $\Omega^{|B|}$ $\varrho \prod_{i \in B} \lambda_i$-a.e. $\prod_{i \in I \backslash B} \sigma_i$-a.e. by \textit{(H1)} of Condition \ref{condhamza} (see \cite[Chapter II, Section 2a)]{MR92}) and hence, $h_j=0$ $\mu_B$-a.e. on $\Lambda_B$. In the case $\delta=1$, we obtain a similar statement for the components of $\nabla_{\Gamma,i} f_k$, $i \in I \backslash B$, $k \in \mathbb{N}$, by considering the term in analogy to (\ref{integral1}) and integrating first with respect to $\prod_{i \in I \backslash B} \sigma_i$ and afterwards with respect to $\prod_{i \in B} \lambda_i$. By Fatou's lemma holds
\[ \mathcal{E}(f_k,f_k) \leq \liminf_{l \rightarrow \infty} \mathcal{E}(f_k-f_{k_l},f_k-f_{k_l}) \rightarrow 0 \quad \text{as } k \rightarrow \infty. \]

\end{proof}

We denote the closure of $(\mathcal{E},\mathcal{D})$ on $L^2(\Lambda;\mu)$ by $(\mathcal{E},D(\mathcal{E}))$.

\begin{proposition}
Suppose that Condition \ref{condL1} and Condition \ref{condhamza} are satisfied. Then $(\mathcal{E},D(\mathcal{E}))$ is a symmetric, regular Dirichlet form.
\end{proposition}

\begin{proof}
By Proposition \ref{propdense} and Lemma \ref{lemmaclosable} $(\mathcal{E},\mathcal{D})$ is symmetric, densely defined and closable with closure $(\mathcal{E},D(\mathcal{E}))$ which is also symmetric. Moreover, by \cite[Chapter I, Prop. 4.10]{MR92} and the representation (\ref{Nform}), $(\mathcal{E},D(\mathcal{E}))$ possesses the Markov property. Finally, $C^{\infty}(\Lambda) \subset C^1(\Lambda) \subset D(\mathcal{E}) \cap C(\Lambda)$ implies that $D(\mathcal{E}) \cap C(\Lambda)$ is dense in $D(\mathcal{E})$ with respect to the $\mathcal{E}_1^{\frac{1}{2}}$-norm as well as in $C(E)$ with respect to the sup-norm. Hence, $(\mathcal{E},D(\mathcal{E}))$ is regular.
\end{proof}

\begin{proposition}
Suppose that Condition \ref{condL1} and Condition \ref{condhamza} are satisfied. Then the symmetric, regular Dirichlet form $(\mathcal{E},D(\mathcal{E}))$ is strongly local and recurrent.
\end{proposition}

\begin{proof}
Using \cite[Theo. 3.1.1]{FOT94} and \cite[Exercise 3.1.1]{FOT94} it is sufficient to show the strong local property for elements in $\mathcal{D}$. Therefore, let $f,g \in \mathcal{D}$ such that $g$ is constant on some open neighborhood $U$ of $\text{supp}(f)$ (in the trace topology of $\Lambda$). Then it follows immediately by (\ref{Nform}) that $\mathcal{E}(f,g)=0$. Hence, $(\mathcal{E},D(\mathcal{E}))$ is strongly local. Clearly, $\mathbbm{1}_{\Lambda} \in \mathcal{D} \subset D(\mathcal{E})$ and $\mathcal{E}(\mathbbm{1}_{\Lambda},\mathbbm{1}_{\Lambda})=0$. Thus, $(\mathcal{E},D(\mathcal{E}))$ is also recurrent.
\end{proof}

We summarize the preceding results in the following theorem:

\begin{theorem}
Assume that Condition \ref{condL1} and Condition \ref{condhamza} are fulfilled. Then the symmetric and positive definite bilinear form $(\mathcal{E},\mathcal{D})$ is densely defined and closable on $L^2(\Lambda;\mu)$. Its closure $(\mathcal{E},D(\mathcal{E}))$ is a recurrent, strongly local, regular, symmetric Dirichlet form on $L^2(\Lambda;\mu)$.
\end{theorem}

By the theory of Dirichlet forms, we obtain immediately the existence of an associated diffusion process. For details see e.g. \cite[Chap. V, Theorem 1.11]{MR92} or \cite[Theorem 7.2.2 and Exercise 4.5.1]{FOT94}. We remark that the definitions of capacities (and hence, of exceptional sets) used in the textbooks \cite{FOT94} and \cite{MR92} are introduced in different ways, but that the defintions coincide in our setting (see \cite[Chap. III, Remark 2.9 and Exercise 2.10]{MR92}). $(T_t)_{t >0}$ denotes the sub-Markovian strongly continuous contraction semigroup on $L^2(\Lambda;\mu)$ corresponding to $(\mathcal{E},D(\mathcal{E}))$.

\begin{theorem} \label{thmdiff}
Suppose that Condition \ref{condL1} and Condition \ref{condhamza} are satisfied. Then there exists a conservative diffusion process (i.e. a strong Markov process with continuous sample paths and infinite life time)
\[ \mathbf{M}:=\big( \mathbf{\Omega}, \mathcal{F}, (\mathcal{F}_t)_{t \geq 0}, (\mathbf{X}_t)_{t \geq 0}, (\Theta_t)_{t \geq 0}, (\mathbf{P}_x)_{x \in \Lambda} \big) \]
with state space $\Lambda$ which is properly associated with $(\mathcal{E},D(\mathcal{E}))$, i.e., for all ($\mu$-versions of) $f \in \mathcal{B}_b(\Lambda) \subset L^2(\Lambda;\mu)$ and all $t >0$ the function
\[ \Lambda \ni x \mapsto p_t f(x):= \mathbb{E}_x \big(f(\mathbf{X}_t) \big) := \int_{\Lambda} f(\mathbf{X}_t) d\mathbf{P}_x \in \mathbb{R} \]
is a quasi continuous version of $T_t f$. $\mathbf{M}$ is up to $\mu$-equivalence unique. In particular, $\mathbf{M}$ is $\mu$-symmetric ($\mu$ is stationary), i.e.,
\[ \int_{\Lambda} p_t f ~ g ~d\mu = \int_{\Lambda} f ~ p_t g ~d\mu \ \text{ for all } f,g \in \mathcal{B}_b(\Lambda) \ \text{ and all } t >0, \]
and has $\mu$ as invariant measure ($\mu$ is reversible), i.e.,
\[ \int_{\Lambda} p_t f~ d\mu = \int_{\Lambda} f~ d\mu \ \text{ for all } f \in \mathcal{B}_b(\Lambda) \ \text{ and all } t >0. \]
\end{theorem}

\begin{remark} \label{remcanonical}
Note that $\mathbf{M}$ is canonical, i.e., $\mathbf{\Omega}=C(\mathbb{R}_{+},\Lambda)$ and $\mathbf{X}_t(\omega)=\omega(t)$, $\omega \in \mathbf{\Omega}$. 
For each $t \geq 0$ we denote by $\Theta_t:\mathbf{\Omega} \rightarrow \mathbf{\Omega}$ the shift operator defined by $\Theta_t(\omega)=\omega(\cdot +t)$ for $\omega \in \mathbf{\Omega}$ such that $\mathbf{X}_s \circ \Theta_t = \mathbf{X}_{s+t}$ for all $s \geq 0$.
We take into account to extend the setting to $C(\mathbb{R}_{+},\mathbb{R}^{Nd})$ by neglecting paths leaving $\Lambda$.
\end{remark}

\subsection{Densities with product structure} \label{secrho}

We introduce a special case of the setting given in Section \ref{secgeneral} which will be of particular importance later on.

\begin{condition} \label{condprod} Assume that $\varrho$ is of the form
\begin{align} \label{density} \varrho(x)=\phi(x) \prod_{i=1}^N \varrho_i (x^i)\quad \text{for } x=(x^1,\dots,x^N) \in \Lambda.
\end{align}
$\varrho_i \in L^1(\overline{\Omega};\lambda + \sigma)$, $i=1,\dots,N$, is given as in (\ref{densityN=1}) for some $\alpha_i \in L^1(\Omega;\lambda)$, $\alpha_i >0$ $\lambda$-a.e. and $\beta_i \in L^1(\Gamma;\sigma)$, $\beta_i >0$ $\sigma$-a.e. such that the respective Hamza conditions are fulfilled (see Section \ref{sec1d}). Moreover, $\phi$ is a $\prod_{i=1}^N (\alpha_i \lambda_i + \beta_i \sigma_i)$-a.e. positive, real valued, measurable function on $\Lambda$ such that $\phi \in L^1(\Lambda; \prod_{i=1}^N (\alpha_i \lambda_i + \beta_i \sigma_i))$. Furthermore, we assume that $\phi$ fulfills Condition \ref{condhamza}.
\end{condition}

\begin{remark}
Note that Condition \ref{condprod} implies Condition \ref{condL1} and Condition \ref{condhamza}.
\end{remark}

Under these conditions it is also possible to consider the form defined in (\ref{Nform}) from a different point of view. Define the form $(\mathcal{E}^i,D(\mathcal{E}^i))$, $i=1,\dots,N$, as the closure of the bilinear form
\begin{align} \label{formdensity}
\frac{1}{2} \int_{\Omega} (\nabla f, \nabla g)~ \alpha_i d\lambda
+ \frac{\delta}{2} \int_{\Gamma} (\nabla_{\Gamma} f,\nabla_{\Gamma} g) ~\beta_i d\sigma \quad \text{for } f,g \in C^1(\overline{\Omega})
\end{align}
on $L^2(\overline{\Omega};\alpha_i \lambda + \beta_i \sigma)$ and set $\mu_i:= \alpha_i \lambda_i + \beta_i \sigma_i$. Then, it is possible to define $(\tilde{\mathcal{E}},\mathcal{D})$ and $(\mathcal{E},\mathcal{D})$ as in (\ref{formsquare}) and (\ref{Nform}) respectively with $\varrho$ replaced by $\phi$. This construction yields the same bilinear form $(\mathcal{E},\mathcal{D})$ on $L^2(\Lambda;\mu)$, where $\mu=\varrho \prod_{i=1}^N (\lambda_i + \sigma_i)$.\\
Roughly speaking, the first definition of $(\mathcal{E},\mathcal{D})$ in Section \ref{secgeneral} corresponds to a Girsanov transformation of $N$ independent sticky Brownian motions on $\overline{\Omega}$ with constant stickyness along $\Gamma$ (each associated to the form $(\mathcal{G},D(\mathcal{G}))$) such that the transformed process has a drift given by $\grad \ln \varrho$. In the present section, the form $(\mathcal{E}^i,D(\mathcal{E}^i))$, $i=1,\dots,N$, describes a \textit{distorted} sticky Brownian motion on $\overline{\Omega}$ with drift $\nabla \ln \alpha_i$ inside $\Omega$ and the stickyness along $\Gamma$ is given by $\frac{\alpha_i}{\beta_i}$ as well as a drift along $\Gamma$ given by $\nabla_{\Gamma} \ln \beta_i$. Then, the Girsanov transformation by $\phi$ yields an additional drift $\nabla \ln \phi$. Note that the resulting form and process (up to equivalence) are the same, since the pre-Dirichlet forms on $C^1(\Lambda)$ coincide.\\
Densities with product structure as presented in the present section have the advantage that we can handle the densities $\varrho_i$, $i=1,\dots,N$, by considering the forms $(\mathcal{E}^i,D(\mathcal{E}^i))$, $i=1,\dots,N$, as given in (\ref{formdensity}). For this type of Dirichlet form it is possible to use the (regularity) results of \cite{GV14b}. In this way the assumptions imposed on $\varrho_i$, $i=1,\dots,N$, are not very restrictive. Only for the interaction part $\phi$ it is necessary to demand stronger requirements.

\section{Analysis of the Markov process} \label{sectanapro} 

\subsection{Generators and boundary conditions}

By Friedrichs representation theorem we have the existence of a unique self-adjoint generator $(L,D(L))$ corresponding to $(\mathcal{E},D(\mathcal{E}))$.

\begin{proposition}
Suppose that Condition \ref{condL1} and Condition \ref{condhamza} are satisfied. Then there exists a unique, self-adjoint, linear operator $(L,D(L))$ on $L^2(\Lambda;\mu)$ such that 
\[ D(L) \subset D(\mathcal{E}) \ \text{ and } ~\mathcal{E}(f,g)=(-Lf,g)_{L^2(\Lambda;\mu)} \text{ for all } f \in D(L),~ g \in D(\mathcal{E}).\] 
\end{proposition}

In order to determine $(L,D(L))$ for a suitable class of functions we need the following additional condition on $\varrho$ and $\Gamma$:

\begin{condition} \label{conddiff}
Assume that $\Gamma$ is $C^2$-smooth. $\varrho$ is given as in Section \ref{secrho}. Moreover, it holds $\phi \in C^1(\Lambda)$ such that $\nabla \ln \phi \in L^2(\Lambda;\mu)$. $\alpha_i, \beta_i \in C(\overline{\Omega})$, $\sqrt{\alpha_i} \in H^{1,2}(\Omega)$ and if $\delta=1$ $\sqrt{\beta_i} \in H^{1,2}(\Gamma)$ for $i=1,\dots,N$.
\end{condition}

\begin{remark}
Note that if $\alpha_i$, $\beta_i$, $i=1,\dots,N$, and $\phi$ are a.e. positive and the additional conditions of Condition \ref{conddiff} are fulfilled, Condition \ref{condL1} and Condition \ref{condhamza} are implied in view of Remark \ref{remcontHamza}. 
\end{remark}

\begin{proposition} \label{propgen}
Suppose that Condition \ref{conddiff} is satisfied and let $f \in C^2(\Lambda)$. Then
\begin{align}
Lf= \sum_{i=1}^N \big( \mathbbm{1}_{\Lambda^{i,\Omega}} ~(L^{i,\Omega} f + L^{i,\Omega}_{\phi} f) + \mathbbm{1}_{\Lambda^{i,\Gamma}} ~(L^{i,\Gamma} f + L^{i,\Gamma}_{\phi} f) \big),
\end{align}
where $L^{i,\Omega} f$, $L^{i,\Gamma} f$ and $L^{i,\phi} f$ for $i=1,\dots,N$ are given by
\begin{align}
L^{i,\Omega} f &= \frac{1}{2} ~\big(\Delta_i f + (\frac{\nabla_i  \alpha_i}{\alpha_i}, \nabla_i f)\big), \\
L^{i,\Gamma} f &= -\frac{1}{2}~ \frac{\alpha_i}{\beta_i}~ (n,\nabla_i f) + \frac{\delta}{2}~ \big( \Delta_{\Gamma,i} f + (\frac{\nabla_{\Gamma,i} \beta_i}{\beta_i}, \nabla_{\Gamma,i} f) \big), \\
L^{i,\Omega}_{\phi} f   &= \frac{1}{2} ~(\frac{\nabla_i \phi}{\phi}, \nabla_i f),\\
L^{i,\Gamma}_{\phi} f   &= \frac{\delta}{2} ~(\frac{\nabla_{\Gamma,i} \phi}{\phi}, \nabla_{\Gamma,i} f) .
\end{align}
\end{proposition}

\begin{proof}
Let $f \in C^2(\Lambda)$ and $g \in \mathcal{D}=C^1(\Lambda)$. By integration by parts and (\ref{density}) follows
\begin{align*}
&\mathcal{E}(f,g)= \frac{1}{2} \int_{\Lambda} \sum_{i=1}^N \big( \mathbbm{1}_{\Lambda^{i,\Omega}} (\nabla_i f,\nabla_i g)+ \delta ~\mathbbm{1}_{\Lambda^{i,\Gamma}} (\nabla_{\Gamma,i} f, \nabla_{\Gamma,i} g) \big) ~\varrho \prod_{j=1}^N (d\lambda_j + d\sigma_j) \\
&=\frac{1}{2} \sum_{i=1}^N \int_{\Lambda^{i,\Omega}} (\nabla_i f,\nabla_i g) ~\varrho \prod_{j=1}^N (d\lambda_j + d\sigma_j) + \frac{\delta}{2} \sum_{i=1}^N \int_{\Lambda^{i,\Gamma}} (\nabla_{\Gamma,i} f,\nabla_{\Gamma,i} g) ~\varrho \prod_{j=1}^N (d\lambda_j + d\sigma_j)\\
&=\frac{1}{2} \sum_{i=1}^N \int_{\overline{\Omega}^{N-1}} \Big( \int_{\Omega} (\nabla_i f,\nabla_i g) ~\varrho~ d\lambda_i \Big) \prod_{j\neq i} (d\lambda_j + d\sigma_j) \\
& \hspace{5cm} +\frac{\delta}{2} \sum_{i=1}^N \int_{\overline{\Omega}^{N-1}} \Big( \int_{\Gamma} (\nabla_{\Gamma,i} f,\nabla_{\Gamma,i} g) ~\varrho~ d\sigma_i \Big) \prod_{j\neq i} (d\lambda_j + d\sigma_j) \\
&=\frac{1}{2} \sum_{i=1}^N \int_{\overline{\Omega}^{N-1}} \Big( - \int_{\Omega} \big( \Delta_i f + (\frac{\nabla_i \varrho}{\varrho}, \nabla_i f) \big) g ~\varrho d\lambda_i + \int_{\Gamma} \frac{\alpha_i}{\beta_i}~ (n, \nabla_i f) g ~\varrho d\sigma_i \Big) \prod_{j\neq i} (d\lambda_j + d\sigma_j)\\
& \hspace{5cm} +\frac{\delta}{2} \sum_{i=1}^N \int_{\overline{\Omega}^{N-1}} \Big(- \int_{\Gamma} (\Delta_{\Gamma,i} f + (\frac{\nabla_{\Gamma,i} \varrho}{\varrho}, \nabla_{\Gamma,i} f) ~\varrho~ d\sigma_i \Big) \prod_{j\neq i} (d\lambda_j + d\sigma_j).
\end{align*}
Note that for $x=(x^1,\dots,x^N) \in \Lambda$ with $x^i \in \Omega$ holds $\frac{\nabla_i \varrho}{\varrho} = \frac{\nabla_i  \alpha_i}{\alpha_i} + \frac{\nabla_i  \phi}{\phi}$ and similarly, if $\delta=1$ and $x^i \in \Gamma$ holds $\frac{\nabla_{\Gamma,i} \varrho}{\varrho}= \frac{\nabla_{\Gamma,i} \beta_i}{\beta_i} + \frac{\nabla_{\Gamma,i}  \phi}{\phi}$ due to (\ref{density}). Hence, 
\[ \mathcal{E}(f,g)= \int_{\Lambda} - \sum_{i=1}^N \big( \mathbbm{1}_{\Lambda^{i,\Omega}} ~(L^{i,\Omega} f + L^{i,\Omega}_{\phi} f) + \mathbbm{1}_{\Lambda^{i,\Gamma}} ~(L^{i,\Gamma} f + L^{i,\Gamma}_{\phi} f) \big) ~g~ d\mu = \int_{\Lambda} -Lf~g~d\mu \]
and therefore, the assertion holds true, since $\mathcal{D}$ is dense in $D(\mathcal{E})$.
\end{proof}

\begin{remark}
In contrast to the case of reflection or absorption, in the present case it is not necessary to require a boundary condition in order to determine the generator, since the reference measure contains the involved surface measures. Nevertheless, it is possible to derive a suitable condition such that the boundary terms vanish. More precisely, if we assume for simplicity that $\delta=0$ and that $f \in C^2(\Lambda)$ is given such that
\begin{align} \label{wentzell} \Delta_i f + (\frac{\nabla_i  \varrho}{\varrho}, \nabla_i f) + \frac{\alpha_i}{\beta_i} (n,\nabla f)=0 \quad \text{on } \Lambda^{i,\Gamma}, ~i=1,\dots,N, 
\end{align}
it holds 
\[ \mathcal{E}(f,g)=-\frac{1}{2} \sum_{i=1}^N \int_{\Lambda} (\Delta_i f + (\frac{\nabla_i \varrho}{\varrho}, \nabla_i f) g d\mu= - \frac{1}{2} \int_{\Lambda} (\Delta f + (\frac{\nabla \varrho}{\varrho}, \nabla f)) g d\mu, \]
i.e., the generator is of the well-known form. The condition (\ref{wentzell}) is called Wentzell boundary condition. Note that the drift in normal direction increases if the factor $\frac{\alpha_i}{\beta_i}$ increases. Hence, it is justifiable to say that the boundary is less sticky for the $i$-th particle at a point $x \in \Gamma$ if $\beta_i(x)$ decreases. This property can also be discovered in a similar way by \cite[Corollary 4.17]{GV14b}, since as a consequence of this ergodicity theorem the particle spends less time on the boundary if $\int_{\Gamma} \beta_i d\sigma$ decreases (compare also to \cite[Corollary 5.7]{FGV14}). Moreover, if we rewrite (\ref{wentzell}) in the form
\begin{align*} \beta_i \Delta_i f +  \beta_i (\frac{\nabla_i \varrho}{\varrho}, \nabla_i f) + \alpha_i (n,\nabla f)=0 \quad \text{on } \Lambda^{i,\Gamma}, 
\end{align*}
and set $\beta_i=0$ we obtain the Neumann boundary condition. 
\end{remark}

Define for $i=1,\dots,N$
\[ A_i :=  \mathbbm{1}_{\Lambda^{i,\Omega}} E + \delta~ \mathbbm{1}_{\Lambda^{i,\Gamma}} P \]
as well as
\begin{align*} 
b_i&:=\frac{1}{2} \Big( \mathbbm{1}_{\Lambda^{i,\Omega}} ~\big( \frac{\nabla_i \alpha_i}{\alpha_i} + \frac{\nabla_i \phi}{\phi} \big) + \mathbbm{1}_{\Lambda^{i,\Gamma}}~ \big( -\frac{\alpha_i}{\beta_i}~n + \delta ~ \frac{\nabla_{\Gamma,i} \beta_i}{\beta_i} + \delta ~\frac{\nabla_{\Gamma,i} \phi}{\phi} \big) \Big),
\end{align*}
where $E$ denotes the $d \times d$ identity matrix. Then, set
\begin{align} \label{Ab} A:=
\begin{pmatrix}
A_1	& 0	& \dots	 & 0      \\
0	& A_2 	& \dots  & 0 	  \\
\vdots	& 0 	& \ddots & \vdots \\
0 	& \dots & 0	 & A_N
\end{pmatrix} \quad
\text{and} \quad b:=
\begin{pmatrix}
b_1  \\
\vdots \\
b_n
\end{pmatrix}
\end{align}
Using this notation, we get for $f \in C^2(\Lambda)$ the representation
\begin{align} \label{repgen} Lf=\frac{1}{2} \text{Tr}(A \nabla^2 f) + (b,\nabla f). 
\end{align}
Note that $AA^t=A^2=A$ and in particular, $P^t=P$.

\subsection{Solution to the martingale problem and SDE} \label{secSDE}

\begin{theorem}
The diffusion process $\mathbf{M}$ from Theorem \ref{thmdiff} is up to $\mu$-equivalence the unique diffusion process having $\mu$ as symmetrizing measure and solving the martingale problem for $(L,D(L))$, i.e., for all $g \in D(L)$
\[ \tilde{g}(\mathbf{X}_t) - \tilde{g}(\mathbf{X}_0) - \int_0^t (Lg)(\mathbf{X}_s) ds, \ t \geq 0, \]
is an $\mathcal{F}_t$-martingale under $\mathbf{P}_x$ for quasi all $x \in \Lambda$. Here $\tilde{g}$ denotes a quasi-continuous version of $g$ (for the definition of quasi-continuity see e.g. \cite[Chap. IV, Proposition 3.3]{MR92}).
\end{theorem}

\begin{proof}
See e.g. \cite[Theorem 3.4 (i)]{AR95}.
\end{proof}

By Proposition \ref{propgen} $L$ is explicitly known on the set $C^2(\Lambda)$. Using the representation given in (\ref{repgen}), we obtain the following corollary:

\begin{corollary}
Assume that Condition \ref{conddiff} is fulfilled. Let $g \in C^2(\Lambda)$ and let $\mathbf{M}$ be the diffusion process from Theorem \ref{thmdiff}. Then
\[ g(\mathbf{X}_t) - g(\mathbf{X}_0) - \int_0^t \frac{1}{2} \text{Tr}(A(\mathbf{X}_s) \nabla^2 g(\mathbf{X}_s)) + (b(\mathbf{X}_s), \nabla g (\mathbf{X}_s)) ds, ~ t \geq 0, \]
is an $\mathcal{F}_t$-martingale under $\mathbf{P}_x$ for quasi every $x \in \Lambda$, where $A$ and $b$ are defined in (\ref{Ab}).
\end{corollary}

\begin{lemma}[weak solutions and martingale problems] \label{lemSV}
Fix the probability measure $\mathbf{P}=\mathbf{P}_x$, $x \in \Lambda$, on $C(\mathbb{R}_{+},\mathbb{R}^{Nd})$ (see also Remark \ref{remcanonical}). Let $A$, $b$ be given on $\Lambda$ by (\ref{Ab}). If 
\[ f(\mathbf{X}_t) - f(\mathbf{X}_0) - \int_0^t \frac{1}{2} \text{Tr} \big( A(\mathbf{X}_s) \nabla^2 f (\mathbf{X_s}) \big) + (b(\mathbf{X}_s), \nabla f(\mathbf{X}_s)) \ ds \]
is an $\mathcal{F}_t$-martingale under $\mathbf{P}$ for every $f \in C^{\infty}_c(\mathbb{R}^d)$, the equation
\[ d\mathbf{X}_t = A(\mathbf{X}_t) dB_t + b(\mathbf{X}_t) dt \]
has a weak solution with distribution $\mathbf{P}$, where $(B_t)_{t \geq 0}$ is an $Nd$-dimensional standard Brownian motion.
\end{lemma}

\begin{proof}
See e.g. \cite[Theorem 18.7]{Kal97}. Note that $A$ and $b$ (defined as in (\ref{Ab})) fulfill the required conditions, since they are progressive. Furthermore, the property $AA^t=A$ is used.
\end{proof}

\begin{remark}
The solution to the SDE given in Lemma \ref{lemSV} results from $\mathbf{M}$ by extending the underlying filtration $(\mathcal{F}_t)_{t \geq 0}$ if necessary (see proof of  \cite[Theorem 18.7]{Kal97} and the references therein). For convenience, we use for the process equipped with the enlarged filtration again the notation
\[ \mathbf{M}=\big( \mathbf{\Omega}, \mathcal{F}, (\mathcal{F}_t)_{t \geq 0}, (\mathbf{X}_t)_{t \geq 0}, (\Theta_t)_{t \geq 0}, (\mathbf{P}_x)_{x \in \Lambda} \big)  \]
taking into account that the associated Dirichlet form is still given by $(\mathcal{E},D(\mathcal{E}))$. Since $(\mathbf{X}_t)_{t \geq 0}$ is $\Lambda$-valued, we use the notation 
\[ \mathbf{X}_t=(\mathbf{X}_t^1,\dots,\mathbf{X}_t^N), ~ t \geq 0, \]
where $\mathbf{X}_t^i$ is $\overline{\Omega}$-valued for $i=1,\dots,N$.
\end{remark}

%Kallenberg p.341

\begin{theorem} \label{thmsolSDE}
$\mathbf{M}$ is a solution to the SDE
\begin{align} \label{SDE} \notag
d\mathbf{X}^i_t =& \mathbbm{1}_{\Omega} (\mathbf{X}^i_t) \Big( dB^i_t + \frac{1}{2} \big( \frac{\nabla_i \alpha_i}{\alpha_i} (\mathbf{X}^i_t) + \frac{\nabla_i \phi}{\phi} (\mathbf{X}_t) \big) dt \Big) - \mathbbm{1}_{\Gamma}(\mathbf{X}^i_t) \frac{\alpha_i}{\beta_i}(\mathbf{X}^i_t) ~n(\mathbf{X}^i_t) dt \\
 + & \delta~\mathbbm{1}_{\Gamma}(\mathbf{X}^i_t) \Big( dB_t^{\Gamma,i} + \big( \frac{\nabla_{\Gamma,i}  \beta_i}{\beta_i} (\mathbf{X}_t^i)+ \frac{\nabla_{\Gamma,i} \phi}{\phi}(\mathbf{X}_t) \big) dt \Big), \quad i=1,\dots,N  \\
dB_t^{\Gamma,i}&=P(\mathbf{X}_t^i) \circ dB_t^i \notag \\ 
\mathbf{X}_0 =& x, \notag
\end{align}
for quasi every starting point $x \in \Lambda$, where $(B_t)_{t \geq 0}$, $B_t=(B_t^1,\dots,B_t^N)$, is an $Nd$-dimensional standard Brownian motion.
\end{theorem}

\begin{remark}
A Fukushima decomposition of $\mathbf{M}$ (see \cite[Chap. 5]{FOT94}) yields the same result as in Theorem \ref{thmsolSDE}. We would like to mention that the argument used here in order to get a solution to the SDE (\ref{SDE}) does not work in this way for reflecting (Neumann) boundary conditions, since in this case the reflection is not given by a drift term. However, a Fukushima decomposition is still valid (see e.g. \cite{Tru03}), because in this case it is also possible to assign an additive functional to the surface measure $\sigma$. The advantage in our situation is that we are able to express the boundary behavior in terms of the generator.
\end{remark}

%\subsection{Ergodicity and occupation time} 

\subsection{Solutions by Girsanov transformations} \label{sectfeller}

\begin{condition} \label{condint}
For every $i=1,\dots,N$, there exists $p_i \geq 2$ with $p_i > \frac{d}{2}$ such that
\[ \frac{|\nabla \alpha_i|}{\alpha_i} \in L^{p_i}_{\text{loc}}(\overline{\Omega} \cap \{ \varrho_i >0\}; \alpha_i \lambda) \quad \text{and additionally }~ \frac{|\nabla_{\Gamma} \beta_i|}{\beta_i} \in L^{p_i}_{\text{loc}}(\Gamma \cap \{ \varrho_i >0\};\beta_i \sigma) ~\text{if } \delta=1 \]
or equivalently
\[ \mathbbm{1}_{\Omega} \frac{|\nabla \alpha_i|}{\alpha_i} + \delta ~\mathbbm{1}_{\Gamma} \frac{|\nabla \beta_i|}{\beta_i} \in L^{p_i}_{\text{loc}}(\overline{\Omega} \cap \{ \varrho_i >0\};\mu_i). \]
Moreover, $\text{cap}_{\mathcal{E}^i}(\{ \varrho_i=0\})=0$.
\end{condition}

Define $\overline{\Omega}_i:=\overline{\Omega} \cap \{\varrho_i >0\}$. Assume that Condition \ref{conddiff} and Condition \ref{condint} are fulfilled.
According to \cite[Theorem 5.9]{GV14b} there exists for every $i=1,\dots,N$ a diffusion process
\[ \mathbf{M}^i:=\big( \mathbf{\Omega}^i, \mathcal{F}^i, (\mathcal{F}^i_t)_{t \geq 0}, (\mathbf{X}^i_t)_{t \geq 0}, (\Theta^i_t)_{t \geq 0}, (\mathbf{P}^i_x)_{x \in \overline{\Omega}_i} \big) \]
with strong Feller transition semigroup $(p_t^i)_{t > 0}$ and transition function $(p^i_t(x,\cdot))_{t > 0}$, $x \in \overline{\Omega}_i$. The processe $\mathbf{M}^i$, $i=1,\dots,N$, is associated to the form $(\mathcal{E}^i,D(\mathcal{E}^i))$ on $L^2(\overline{\Omega}_i;\mu_i)$, where $\mu_i=\alpha_i \lambda +\beta_i \sigma$. In particular, $(p^i_t)_{t > 0}$ is absolutely continuous with respect to $\mu_i$, i.e., for every $t >0$ and $x \in \overline{\Omega}_i$, there exists a non-negative, measurable function $p^i_t(x,y)$, $y \in \overline{\Omega_i}$, such that
\[ p^i_t(x,A)=\int_A p^i_t(x,y) d\mu_i(y) \quad \text{for every } A \in \mathcal{B}(\overline{\Omega}_i). \] 
Let $\mathbf{M}$ be given by
\[  \mathbf{M}:=\big( \times_{i=1}^N \mathbf{\Omega}^i, \otimes_{i=1}^N \mathcal{F}^i, (\otimes_{i=1}^N \mathcal{F}^i_t)_{t \geq 0}, (\mathbf{X}_t)_{t \geq 0}, (\Theta_t)_{t \geq 0}, (\otimes_{i=1}^N \mathbf{P}^i_{x^i})_{x=(x^1,\dots,x^N) \in \tilde{\Lambda}} \big), \]
where $\tilde{\Lambda}:=\times_{i=1}^N \overline{\Omega}_i$ as well as 
\[ \mathbf{X}_t(\omega):=(\mathbf{X}^1_t(\omega_1),\dots,\mathbf{X}^N_t(\omega_N)) \quad \text{and } ~ \Theta_t(\omega):=(\Theta^1_t(\omega_1),\dots,\Theta^N_t(\omega_N)) \]
for $\omega=(\omega_1,\dots,\omega_N) \in \times_{i=1}^N \mathbf{\Omega}^i$. Set $\mathbf{P}_x:=\otimes_{i=1}^N \mathbf{P}^i_{x^i}$ for $x=(x^1,\dots,x^N) \in \tilde{\Lambda}$.\\
Denote by $(p_t)_{t > 0}$ the transition semigroup and by $(p_t(x,\cdot))_{t > 0}$, $x \in \tilde{\Lambda}$, the transition function of $\mathbf{M}$. Then, it holds for every $A=A_1 \times \dots \times A_N \in \times_{i=1}^N \mathcal{B}(\overline{\Omega_i}) \subset \mathcal{B}(\tilde{\Lambda})$ 
\begin{align*} p_t(x,A) &= \int_{\times_{i=1}^N \mathbf{\Omega}^i} \mathbbm{1}_A (\mathbf{X}_t(\omega))~ d\mathbf{P}_x(\omega) \\
&=\int_{\times_{i=1}^N \mathbf{\Omega}^i} \prod_{i=1}^N~\mathbbm{1}_{A_i} (\mathbf{X}^i_t(\omega_i))~ d\mathbf{P}_x(\omega) \\
&=\prod_{i=1}^N \int_{\mathbf{\Omega}^i} \mathbbm{1}_{A_i}(\mathbf{X}^i_t(\omega_i)) ~d\mathbf{P}^i_{x^i}(\omega_i) =\prod_{i=1}^N p_t^i(x^i,A_i)
\end{align*}
by definition of $\mathbf{P}_x$. Since $\times_{i=1}^N \mathcal{B}(\overline{\Omega_i})$ generates $\mathcal{B}(\tilde{\Lambda})$, it holds
\[ p_t(x,A)= \int_A \prod_{i=1}^N p_t^i(x^i,y^i) \prod_{i=1}^N d\mu_i(y^i) \quad \text{for every } A \in \mathcal{B}(\tilde{\Lambda}). \]
As a consequence, $p_t(x,\cdot)$, $t > 0$, $x \in \tilde{\Lambda}$, is absolutely continuous with respect to $\prod_{i=1}^N \mu_i$ and 
\begin{align} p_tf(x^1,\dots,x^N)=\hat{p}_t^N \dots \hat{p}_t^1 f(x^1,\dots,x^N) \quad \text{for every } f \in \mathcal{B}_b(\tilde{\Lambda}), \label{transSG}
\end{align}
where 
\[\hat{p}_t^if(x^1,\dots,x^N):=p^i_t f(x^1,\dots,x^{i-1},\cdot,x^{i+1},\dots,x^N)(x^i)\] 
and the order of the $\hat{p}_t^i$, $i=1,\dots,N$, is arbitrary.\\

Consider the symmetric bilinear form on $L^2(\tilde{\Lambda};\prod_{i=1}^N \mu_i)$ given by
\[ \big( \otimes_{i=1}^N \mathcal{E}^i \big)(f,g):=\sum_{i=1}^N \int_{\times_{j \neq i} \overline{\Omega}_j} \mathcal{E}^i(f,g) \prod_{j \neq i} d\mu_j ,\]
where
\begin{align*}
f,g \in &D(\otimes_{i=1}^N \mathcal{E}^i):=\{ f \in L^2(\tilde{\Lambda};\prod_{i=1}^N \mu_i)\big|~\text{for each } i=1,\dots,N \text{ and for } \prod_{j \neq i} \mu_j-{a.e. } \\
&(x^1,\dots ,x^{i-1},x^{i+1},\dots ,x^N) \in \times_{j \neq i} \overline{\Omega}_j: f(x^1,\dots, x^{i-1},\cdot,x^{i+1},\dots,x^n) \in D(\mathcal{E}^i) \}.
\end{align*}
Due to \cite[Chapter V, Section 2.1]{BH91} $(\otimes_{i=1}^N \mathcal{E}^i,D(\otimes_{i=1}^N \mathcal{E}^i))$ is a Dirichlet form on $L^2(\tilde{\Lambda};\prod_{i=1}^N \mu_i)$. Obviously, this Dirichlet form extends the pre-Dirichlet form $(\tilde{\mathcal{E}},\mathcal{D})$ defined in (\ref{formbouleau}).

\begin{lemma} \label{lemdense}
$C^1(\Lambda)$ is dense in $D(\otimes_{i=1}^N \mathcal{E}^i)$ w.r.t. $\big( \otimes_{i=1}^N \mathcal{E}^i \big)_1^{\frac{1}{2}}$, i.e., $(\otimes_{i=1}^N \mathcal{E}^i,D(\otimes_{i=1}^N \mathcal{E}^i))$ is the closure of $(\tilde{\mathcal{E}},\mathcal{D})$ on $L^2(\tilde{\Lambda};\prod_{i=1}^N \mu_i)$.
\end{lemma}

\begin{proof}
First, note that $C^1(\Lambda) \subset D(\otimes_{i=1}^N \mathcal{E}^i)$ by definition of $D(\otimes_{i=1}^N \mathcal{E}^i)$. For simplicity, we only consider the case $N=2$. The statement for arbitrary $N \in \mathbb{N}$ follows by the same arguments. By \cite[Proposition 2.1.3b)]{BH91} $D(\mathcal{E}^1) \otimes D(\mathcal{E}^2)$ is dense in $D(\mathcal{E}^1 \otimes \mathcal{E}^2)$. Thus, it is sufficient to show that $C^1(\Lambda)$ is dense in $D(\mathcal{E}^1) \otimes D(\mathcal{E}^2)$. Then the assertion follows by a diagonal sequence argument. Let $h \in D(\mathcal{E}^1) \otimes D(\mathcal{E}^2)$ such that $h(x^1,x^2)=f(x^1)g(x^2)$ for $\prod_{i=1}^2 \mu_i$-a.e. $(x^1,x^2) \in \tilde{\Lambda}$, where $f \in D(\mathcal{E}^1)$ and $g \in D(\mathcal{E}^2)$. Since $C^1(\overline{\Omega})$ is dense in $D(\mathcal{E}^1)$ and $D(\mathcal{E}^2)$, we can choose sequences $(f_k)_{k \in \mathbb{N}}$ and $(g_k)_{k \in \mathbb{N}}$ in $C^1(\overline{\Omega})$ such that $f_k \rightarrow f$ in $D(\mathcal{E}^1)$ and $g_k \rightarrow g$ in $D(\mathcal{E}^2)$ as $k \rightarrow \infty$. Define $h_k(x^1,x^2):=f_k(x^1)g_k^(x^2)$ for $x^1,x^2 \in \overline{\Omega}$. Then it follows easily by the prodcut structure of the underlying measure that the sequence $(h_k)_{k \in \mathbb{N}}$, $h_k \in C^1(\Lambda)$, converges in $L^2(\tilde{\Lambda};\prod_{i=1}^2 \mu_i)$ to $h$ and moreover, the sequence is $\mathcal{E}^1 \otimes \mathcal{E}^2$-Cauchy.
\end{proof}

Denote by $(T_t^i)_{t > 0}$ the $L^2(\overline{\Omega}_i;\mu_i)$-semigroup of $(\mathcal{E}^i,D(\mathcal{E}^i))$, $i=1,\dots,N$. By \cite[Chapter V, Proposition 2.1.3]{BH91} the $L^2(\tilde{\Lambda};\prod_{i=1}^N \mu_i)$-semigroup $(T_t)_{t > 0}$ associated to $(\otimes_{i=1}^N \mathcal{E}^i,D(\otimes_{i=1}^N \mathcal{E}^i))$ is given by
\[ T_tf =\hat{T}_t^{N} \cdots \hat{T}_t^{1} f \quad \text{for } f \in L^2(\tilde{\Lambda};\prod_{i=1}^N \mu_i),\]
where 
\[ \hat{T}_t^i f(x^1,\dots,x^n):= T_t^i f(x^1,\dots,x^{i-1},\cdot,x^{i+1},\dots,x^n)(x^i)
\]
for $x=(x^1,\dots,x^n) \in \tilde{\Lambda}$. Since $\mathbf{M}^i$ is associated to the form $(\mathcal{E}^i,D(\mathcal{E}^i))$ for $i=1,\dots,N$, it follows by (\ref{transSG}) and Lemma \ref{lemdense} the following:

\begin{proposition}
The Dirichlet form associated to $\mathbf{M}$ is given by the closure of $(\tilde{\mathcal{E}},C^1(\Lambda))$ on $L^2(\tilde{\Lambda};\prod_{i=1}^N \mu_i)$.
\end{proposition}

Additionally to Condition \ref{conddiff} and Condition \ref{condint} we assume the following:

\begin{condition} \label{condpos}
$\phi$ is strictly positive.
\end{condition}

Under these conditions on $\phi$ it is possible to perform a Girsanov transformation of $\mathbf{M}$. Consider the multiplicative functional $(Z_t)_{t \geq 0}$, $Z_t=\exp(M_t- \frac{\langle M \rangle_t}{2})$, given by
\[ M_t:= \int_0^t \nabla \ln \phi(\mathbf{X}_t) dB_t, \ \ t \geq 0.\]
Note that $\nabla \ln \phi(\mathbf{X}_t)=\frac{\nabla \phi}{\phi}(\mathbf{X}_t)$ and $B_t$, $t \geq 0$, are $\mathbb{R}^{Nd}$ valued and also that $\nabla \ln \phi$ is bounded due to Condition \ref{condpos}.\\

In view of Remark \ref{remcanonical} (applied to $(\mathcal{E}^i,D(\mathcal{E}^i))$, $i=1,\dots,N$), it holds $\times_{i=1}^N \mathbf{\Omega}^i=C(\mathbb{R}_+, \tilde{\Lambda})$ and  $\otimes_{i=1}^N \mathcal{F}^i=\mathcal{B}(C(\mathbb{R}_+, \tilde{\Lambda}))$. Thus, $(\times_{i=1}^N \mathbf{\Omega}^i,\otimes_{i=1}^N \mathcal{F}^i)$ is a standard measurable space (see \cite[Chapter I, Definition 3.3]{IW89}) and hence, by \cite[Chapter IV, Section 4]{IW89} there exists for every $x \in \tilde{\Lambda}$ a probability measure $\mathbf{P}_x^{\phi}$ such that $\big(\mathbf{P}_x^{\phi}\big)_{|\otimes_{i=1}^N \mathcal{F}^i_t}= \mathbf{P}^{\phi}_{x,t}$, where \[ \mathbf{P}^{\phi}_{x,t}(A):=\int_A Z_t(\omega) d\mathbf{P}_x(\omega) \quad \text{for } A \in \otimes_{i=1}^N \mathcal{F}^i_t.\]
Let
\[  \mathbf{M}^{\phi}:=\big( \times_{i=1}^N \mathbf{\Omega}^i, \otimes_{i=1}^N \mathcal{F}^i, (\otimes_{i=1}^N \mathcal{F}^i_t)_{t \geq 0}, (\mathbf{X}_t)_{t \geq 0}, (\Theta_t)_{t \geq 0}, (\mathbf{P}_x^{\phi})_{x \in \tilde{\Lambda}} \big). \]
Then, the transition function $(p^{\phi}_t(x,\cdot))_{t >0}$ of $\mathbf{M}^{\phi}$ is absolutely continuous with respect to $\mu$ for every $x \in \tilde{\Lambda}$. Indeed, by the previous considerations the transition function $(p_t(x,\cdot))_{t >0}$ is absolutely continuous with respect to $\prod_{i=1}^N \mu_i$. Assume that $A \in \mathcal{B}(\tilde{\Lambda})$ is given such that $\mu(A)=0$. Since $\phi$ is bounded from above and from below away from zero in view of Condition \ref{condpos} and the continuity of $\phi$, it also holds that $\prod_{i=1}^N \mu_i (A)=0$ and hence, $p_t(x,A)=0$ for every $t>0$ and $x \in \tilde{\Lambda}$, i.e., 
\[ \int_{\times_{i=1}^N \mathbf{\Omega}^i} \mathbbm{1}_A(\mathbf{X}_t) ~d\mathbf{P}_x =0 \quad \text{for every } t>0 \text{ and } x \in \tilde{\Lambda}.\] Therefore, we also have
\begin{align*}
p_t^{\phi}(x,A)=\int_{\times_{i=1}^N \mathbf{\Omega}^i} \mathbbm{1}_A(\mathbf{X}_t)~ d\mathbf{P}_x^{\phi}=\int_{\times_{i=1}^N \mathbf{\Omega}^i} \mathbbm{1}_A(\mathbf{X}_t)~ d\mathbf{P}_{x,t}^{\phi}=\int_{\times_{i=1}^N \mathbf{\Omega}^i} Z_t ~\mathbbm{1}_A(\mathbf{X}_t) ~d\mathbf{P}_x =0
\end{align*}\\
We summarize the results of this section in the following theorem:

\begin{theorem} \label{thmgirsanov} $\mathbf{M}^{\phi}$ is a solution to the SDE
\begin{align} \label{SDE2} \notag
d\mathbf{X}^i_t =& \mathbbm{1}_{\Omega} (\mathbf{X}^i_t) \Big( dB^i_t + \frac{1}{2} \big( \frac{\nabla_i \alpha_i}{\alpha_i} (\mathbf{X}^i_t) + \frac{\nabla_i \phi}{\phi} (\mathbf{X}_t) \big) dt \Big) - \mathbbm{1}_{\Gamma}(\mathbf{X}^i_t) \frac{\alpha_i}{\beta_i}(\mathbf{X}^i_t) ~n(\mathbf{X}^i_t) dt \\
 + & \delta~\mathbbm{1}_{\Gamma}(\mathbf{X}^i_t) \Big( dB_t^{\Gamma,i} + \big( \frac{\nabla_{\Gamma,i} \beta_i}{\beta_i} (\mathbf{X}_t^i)+ \frac{\nabla_{\Gamma,i} \phi}{\phi}(\mathbf{X}_t) \big) dt \Big), \quad i=1,\dots,N  \\
dB_t^{\Gamma,i}&=P(\mathbf{X}_t^i) \circ dB_t^i \notag \\ 
\mathbf{X}_0 =& x, \notag
\end{align}
for every starting point $x \in \tilde{\Lambda}$, where $(B_t)_{t \geq 0}$, $B_t=(B_t^1,\dots,B_t^N)$, is an $Nd$-dimensional standard Brownian motion. Moreover, the Dirichlet form associated to $\mathbf{M}^{\phi}$ is given by $(\mathcal{E},D(\mathcal{E}))$ on $L^2(\tilde{\Lambda};\mu)$ and its transition function $(p_t^{\phi}(x,\cdot))_{t > 0}$ is absolutely continuous with respect to $\mu$ for every $x \in \tilde{\Lambda}$.
\end{theorem}

\begin{proof}
Due to results in \cite{GV14b} every $\mathbf{M}^i$ solves the respective $d$-dimensional SDE for every starting point in $\overline{\Omega}_i$, $i=1,\dots,N$. Hence, the process $\mathbf{M}$ solves the SDE for $N$ independent particles, i.e., it solves (\ref{SDE2}) for $\phi$ given by the indicator function on $\Lambda$. As a consequence $\mathbf{M}^{\phi}$ solves (\ref{SDE2}) by the Girsanov transformation theorem (see \cite[Chapter IV, Section 4]{IW89}). Moreover, by the same arguments as in \cite{GV14a} the Dirichlet form of the transformed process $\mathbf{M}^{\phi}$ is given by $(\mathcal{E},D(\mathcal{E}))$.
\end{proof}

\subsection{Connection to random time changes}

In the following, we present the connections to random time changes for the case $\delta=0$ and in particular, how the Dirichlet form construction is related to it.\\

As already mentioned in \cite{GV14b} the sticky boundary behavior is strongly related to random time changes. Denote by $\mathbf{Y}=(\mathbf{Y}_t)_{t \geq 0}$ reflecting Browian motion on $\overline{\Omega}$. The associated SDE is given by
\[ d\mathbf{Y}_t=dB_t + \frac{1}{2} dL_t^{\mathbf{Y}}, \]
where $(L_t^{\mathbf{Y}})_{t \geq 0}$, $L^{\mathbf{Y}}_t=- \int_0^t n(\mathbf{Y}_s) dl_s^{\mathbf{Y}}$, denotes the boundary local time of $\mathbf{Y}$. In this case, we also refer to this kind of boundary behavior as instantaneous reflection, since the process does not spend time on the boundary $\Gamma$. The underlying Dirichlet form is given by
\[ \frac{1}{2} \int_{\overline{\Omega}} (\nabla f,\nabla g) ~d\lambda \quad \text{for } f,g \in H^{1,2}(\Omega) \]
on $L^2(\overline{\Omega};\lambda)$ and $(l_t^{\mathbf{Y}})_{t \geq 0}$ is the additive functional in Revuz correpsondence with the surface measure $\sigma$ on $\Gamma$. Note that $\sigma$ is in this case singular with respect to the reference measure $\lambda$ in the sense that the support of $\sigma$ has measure zero with respect to $\lambda$ and this singularity describes the instantaneous reflection. Denote by $(T_t)_{t \geq 0}$ the inverse of the additive functional $A_t:=t + \beta(\mathbf{Y}_t) l_t^{\mathbf{Y}}$, $t \geq 0$, where $\beta \in C(\overline{\Omega})$ is strictly positive. Using $(T_t)_{t \geq 0}$ as a new time scale it is possible to perform a random time change, namely we define the process $\mathbf{X}=(\mathbf{X}_t)_{t \geq 0}$ by $\mathbf{X}_t:=\mathbf{Y}_{T_t}$ for $t \geq 0$. Then, similar to the case of the positive half-line in \cite{EP14} we obtain with the definition $l_t^{\mathbf{X}}:=l_{T_t}^{\mathbf{Y}}$, $t \geq 0$,
\[ \mathbf{X}_t - \mathbf{X}_0= \mathbf{Y}_{T_t} -\mathbf{Y}_0= B_{T_t} + \frac{1}{2} L_{T_t}^{\mathbf{Y}}=B_{T_t} + \frac{1}{2} L_{t}^{\mathbf{X}}, \]
where $L_{t}^{\mathbf{X}}=-\int_0^t n(\mathbf{X_s}) dl_s^{\mathbf{X}}$, and moreover,
\begin{align*} \langle B_{T_t}, B_{T_t} \rangle = T_t= \int_0^{T_t} \mathbbm{1}_{\Omega}(\mathbf{Y}_s) ds &= \int_0^{T_t} \mathbbm{1}_{\Omega}(\mathbf{Y}_s) dA_s \\
&= \int_0^{t} \mathbbm{1}_{\Omega}(\mathbf{Y}_{T_s}) ds = \int_0^{t} \mathbbm{1}_{\Omega}(\mathbf{X}_s) ds
\end{align*} 
Thus, by eventually enlarging the filtered probability space there exists a standard Brownian motion $(\tilde{B}_t)_{t \geq 0}$ such that 
\[ \mathbf{X}_t - \mathbf{X}_0= \int_0^{t} \mathbbm{1}_{\Omega}(\mathbf{X}_s) d\tilde{B}_s + \frac{1}{2} L_t^{\mathbf{X}}. \]
Furthermore, it holds
\begin{align*}
\int_0^t \frac{1}{\beta(\mathbf{X}_s)}~\mathbbm{1}_{\Gamma}(\mathbf{X}_s) ds &=\int_0^t \frac{1}{\beta(\mathbf{Y}_{T_s})}~ \mathbbm{1}_{\Gamma}(\mathbf{Y}_{T_s}) dA_{T_s} \\
&=\int_0^{T_t} \frac{1}{\beta(\mathbf{Y}_{s})} \mathbbm{1}_{\Gamma}(\mathbf{Y}_s) dA_s =\int_0^{T_t} \mathbbm{1}_{\Gamma}(\mathbf{Y}_s) dl_s^{\mathbf{Y}} = l_{T_t}^{\mathbf{Y}}=l_t^{\mathbf{X}},
\end{align*}
i.e., $\beta(\mathbf{X}_t) dl_t^{\mathbf{X}}= \mathbbm{1}_{\Gamma}(\mathbf{X}_t) dt$ and in particular, $L_{t}^{\mathbf{X}}=- \int_0^t \frac{1}{\beta(\mathbf{X}_s)}~ \mathbbm{1}_{\Gamma}(\mathbf{X}_s) n(\mathbf{X}_s) ds$. As a consequence $(\mathbf{X}_t)_{t \geq 0}$ solves the SDE
\[ d\mathbf{X}_t= \mathbbm{1}_{\Omega}(\mathbf{X}_t) dB_t - \frac{1}{2} \frac{1}{\beta(\mathbf{X}_t)}~ \mathbbm{1}_{\Gamma}(\mathbf{X}_t) n(\mathbf{X}_t) dt. \]
According to the theory presented in \cite{CF11} (see also \cite{GV14b}) the Dirichlet form corresponding to the time changed process is given by the closure of
\[ \frac{1}{2} \int_{\overline{\Omega}} (\nabla f,\nabla g) d\lambda=\frac{1}{2} \int_{\overline{\Omega}} (\nabla f,\nabla g)~ \mathbbm{1}_{\Omega} ( d\lambda + \beta d\sigma) \quad \text{for } f,g \in C^1(\overline{\Omega}) \ \text{ on } L^2(\overline{\Omega};\lambda + \beta \sigma). \] 
If we choose e.g. $\alpha \in C^1(\overline{\Omega})$ positive, a drift transformation by $\frac{1}{2} \nabla \ln \alpha$ yields a solution to
\[ d\mathbf{X}_t= \mathbbm{1}_{\Omega}(\mathbf{X}_t) dB_t + \frac{1}{2} \mathbbm{1}_{\Omega}(\mathbf{X}_t) \nabla \ln \alpha(\mathbf{X}_t) dt - \frac{1}{2} \frac{\alpha(\mathbf{X}_t}{\beta(\mathbf{X}_t)}~ \mathbbm{1}_{\Gamma}(\mathbf{X}_t) n(\mathbf{X}_t) dt, \]
which is associated to the closure of 
\[ \frac{1}{2} \int_{\overline{\Omega}} (\nabla f,\nabla g) \alpha d\lambda=\frac{1}{2} \int_{\overline{\Omega}} (\nabla f,\nabla g)~ \mathbbm{1}_{\Omega} ( \alpha d\lambda + \beta d\sigma) \quad \text{for } f,g \in C^1(\overline{\Omega}) \ \text{ on } L^2(\overline{\Omega};\alpha \lambda + \beta \sigma). \]
Actually, the order of time change and Girsanov transformation does not matter in this case. This construction yields a single particle diffusing in $\overline{\Omega}$ with sticky boundary behavior. Consequently, the idea in order to construct an interacting particle system with sticky boundary is to consider $N$ independent particles in $\Omega$ which are connected to the tensor product of the forms $(\mathcal{E}^i,D(\mathcal{E}^i))$, $i=1,\dots, N$, as presented above. Afterwards, a drift is introduced by the density $\phi$ which finally leads to the form considered in the present paper. \\
An evident idea would also be to construct an interacting particle system with instantaneous reflection and to realize afterwards a time change. However, this is not possible in a simple way. The canonical Dirichlet form is given by the closure of 
\begin{align} \label{reflform} \frac{1}{2} \int_{\Lambda} (\nabla f, \nabla g) \varrho d\lambda^N \quad \text{for } f,g \in C^1(\Lambda) \ \text{on } L^2(\Lambda;\varrho \lambda^N),
\end{align}
where $\lambda^N$ denotes the Lebesgue measure on $\Lambda$. For this kind of Dirichlet form we have a well-known regularity theory at hand which enables us to construct solutions to the underlying SDE even for singular drifts for every starting point in a specified set of admissible initial values (see e.g. \cite{FG08}, \cite{BG12} and \cite{FT95}). Usually, only starting points in the corners  of $\Lambda$ (two or more particles at the boundary of $\Omega$) are not admissible, since the boundary is not sufficiently smooth at these points. Nevertheless, such kind of dynamics do not diffuse on the boundary of $\Lambda$ and hence, a time changed process will also not have this property. Therefore, it is not possible to construct an interacting particle system with sticky reflection via time change in use of the closure of (\ref{reflform}), since a particle which reaches $\Gamma$ is expected to sojourn a positive amount of time on $\Gamma$ and meanwhile, the remaining particles keep on moving undelayed. This implies a diffusion on the boundary of $\Lambda$. An appropiate approach for a process with boundary diffusion and instantaneous reflection is given in \cite{Tom80} and \cite{Car09}. \\
In \cite{GV14b} it is shown that the transition semigroup of $\mathbf{M}^i$, $i=1,\dots,N$, given above has the strong Feller property. It seems not clear that the transition semigroup of $\mathbf{M}$ is doubly Feller (i.e., it is a Feller process with strong Feller transition semigroup). In this case, it would even be possible to deduce due to the results of \cite{Chu85} and \cite{CK08} that the process $\mathbf{M}^{\phi}$ of Theorem \ref{thmgirsanov} has the doubly Feller property.

\section{Application to particle systems with singular interactions} \label{sectappl} 

In \cite{Gra88} the author investigates a martingale problem with Wentzell boundary conditions in a very general form. In particular, the relation to SDEs is developed and an existence result is shown. As an application the author constructs a system of interacting particles in a domain with sticky boundary. This particle system gives a model for particles diffusing in a chromatography tube. More precisely, the considered domain is given by $\Theta:=\{ x \in \mathbb{R}^d|~x_1 >0 \}$ and the investigated SDE on $\overline{\Theta}$ reads as follows:
\begin{align*} 
d\mathbf{X}_t &= \sigma(\mathbf{X}_t) dN_t + b(\mathbf{X}_t) (dt - \rho(\mathbf{X}_t) dK_t) + \gamma(\mathbf{X}_t) dK_t + \tau(\mathbf{X}_t) dC_{K_t}, \\
\mathbf{X}_0&=x \in \overline{\Theta},
\end{align*}
where $(\mathbf{X}_t)_{t \geq 0}$ is a continuous, $\overline{\Theta}$-valued process, $(C_t)_{t \geq 0}$ is a $d$-dimensional standard Brownian motion, $(N_t)_{t \geq 0}$ is a $d$-dimensional continuous martingale and $(K_t)_{t \geq 0}$ is given such that $K_0=0$, $K_t$ is increasing, $dK_t=\mathbbm{1}_{\partial \Theta}(\mathbf{X}_t) dK_t$, and
\[ \langle N^i, N^j \rangle_t=\delta_{ij} \big( t- \int_0^t \rho(\mathbf{X}_s) dK_s \big). \]
Here, the main focus is placed on the very general form of the martingale problem and SDE as well as the assumptions on $\sigma$ and $a=\sigma \sigma^{T}$, which is not necessarily strictly elliptic. In former results (see e.g. \cite[Chapter IV, Section 7]{IW89}), it is assumed amongst other things that $a_{11} \geq c >0$. 
In \cite{Gra88} it is shown that the martingale problem with the sojourn condition $\rho(\mathbf{X}_t) dK_t \leq \mathbbm{1}_{\partial \Theta}(\mathbf{X}_t)dt$ has a solution if and only if the above SDE has a weak solution. Sufficient conditions are $\tau =0$, $\sigma$ and $b$ are uniformly Lipschitz continuous and bounded,  $\gamma=n$ is the inward normal vector and $\rho$ is bounded, measurable and positive. Nevertheless, the smoothness conditions on $b$ are rather strong. If we assume additionally  that $a_{11}>0$ (e.g. if $\sigma$ is given by the identity matrix), it holds that
\[ \rho(\mathbf{X}_t) dK_t= \mathbbm{1}_{\partial \Theta}(\mathbf{X}_t)dt. \]
In the case of the identity matrix, the underlying SDE is given by
\begin{align*} 
d\mathbf{X}_t &= \mathbbm{1}_{\Theta}(\mathbf{X}_t) dB_t + b(\mathbf{X}_t) \mathbbm{1}_{\Theta}(\mathbf{X}_t )dt + \frac{1}{\rho(\mathbf{X}_t)}~n(\mathbf{X}_t) dt, \\
\mathbf{X}_0&=x \in \overline{\Theta},
\end{align*}
where $(B_t)_{t \geq 0}$ is a $d$-dimensional standrad Brownian motion. This setting corresponds to the one considered in \cite{GV14b} for $\delta=0$. The corresponding system of interacting particles is given by
\begin{align*} 
d\mathbf{X}^i_t &= \mathbbm{1}_{\Theta}(\mathbf{X}^i_t) dB^i_t + b^i(\mathbf{X}_t) \mathbbm{1}_{\Theta}(\mathbf{X}^i_t )dt + \frac{1}{\rho^i(\mathbf{X}_t)}~n(\mathbf{X}^i_t) dt, \quad i=1,\dots,N, \\
\mathbf{X}_0&=x \in \overline{\Theta}^N,
\end{align*}
where $\mathbf{X}_t=(\mathbf{X}_t^1,\dots,\mathbf{X}_t^N)$. According to \cite{Gra88} an application for this system of SDEs is a model for molecules diffusing in a chromatography tube. The particles are pushed by a flow of gas and are absorbed and released by a liquid state deposited on the boundary of the tube. Hence, it is resonable to suppose a sticky boundary behavior. However, it is physically unsreasonable that two molecules are located at the same position in $\overline{\Theta}$ at the same time. In order to avoid this kind of behavior it is necessary to consider a singular drift $b^i$, $i=1,\dots,N$, which causes a strong repulsion if two particles get close to each other. The construction of such kind of stochastic dynamics via Dirichlet forms has already been realized for absorbing and reflecting boundary conditions.\\

In analogy to \cite[Section 5]{FG08}, a continuous pair potential (without hard core) is a continuous function $\zeta: \mathbb{R}^d \rightarrow \mathbb{R} \cup \{ \infty \}$ such that $\zeta(-x)=\zeta(x) \in \mathbb{R}$ for all $x \in \mathbb{R}^d \backslash \{ 0 \}$. $\zeta$ is said to be repulsive if there exists a continuous decreasing function $\eta:(0,\infty) \rightarrow [0,\infty)$ with $\lim_{t \rightarrow 0} \eta(t)=\infty$ and $R>0$ such that 
\[ \zeta(x) \geq \eta(|x|) \quad \text{for } |x| \leq R. \]
In particular, $\zeta(0)=\infty$.
For $N \in \mathbb{N}$ and a repulsive continuous pair potential $\zeta$ we consider the the function
\[ \phi(x):= \exp( - \sum_{1\leq i,j \leq N \atop i\ne j} \zeta(x^i-x^j)) \quad \text{for } x=(x^1,\dots,x^N) \in \Lambda=\overline{\Omega}^N. \] 
Note that $\phi(x)=0$ if there exist $i,j \in \{1,\dots,N\}$ such that $x^i=x^j$.\\
Let $\Gamma$ be $C^2$-smooth. We assume that $\zeta$ is a repulsive, continuous pair potential such that $\phi \in C^1(\Lambda)$ and moreover, we assume that
\[ \nabla \ln \phi \in L^2(\Lambda;\mu) \quad \text{with } \mu=\phi ~ \prod_{i=1}^N (\alpha_i \lambda_i + \beta_i \sigma_i), \]
where $\alpha_i$ and $\beta_i$ are continuous and a.e. positive such that $\sqrt{\alpha_i} \in H^{1,2}(\Omega)$ and $\sqrt{\beta_i} \in H^{1,2}(\Gamma)$ for $i=1,\dots,N$. Then Condition \ref{conddiff} is fulfilled and Theorem \ref{thmsolSDE} can be applied, i.e., there exists a solution to the SDE
\begin{align*}
d\mathbf{X}^i_t =& \mathbbm{1}_{\Omega} (\mathbf{X}^i_t) \Big( dB^i_t + \frac{1}{2} \big( \frac{\nabla_i \alpha_i}{\alpha_i} (\mathbf{X}^i_t) - \sum_{j \neq i} \nabla_i \zeta(\mathbf{X}_t^i -\mathbf{X}_t^j) \big) dt \Big) - \mathbbm{1}_{\Gamma}(\mathbf{X}^i_t) \frac{\alpha_i}{\beta_i}(\mathbf{X}^i_t) ~n(\mathbf{X}^i_t) dt \\
 + & \delta~\mathbbm{1}_{\Gamma}(\mathbf{X}^i_t) \Big( dB_t^{\Gamma,i} + \big( \frac{\nabla_{\Gamma,i} \beta_i}{\beta_i} (\mathbf{X}_t^i)- \sum_{j \neq i} \nabla_{\Gamma,i} \zeta(\mathbf{X}_t^i -\mathbf{X}_t^j) \big) dt \Big), \quad i=1,\dots,N  \\
dB_t^{\Gamma,i}&=P(\mathbf{X}_t^i) \circ dB_t^i  \\ 
\mathbf{X}_0 =& x, 
\end{align*}
for quasi every starting point $x \in \Lambda$.

\begin{example} A possible example is given by the Lennard-Jones potential
\[ \zeta(x)= 4 \varepsilon~ \big( (\frac{c}{|x|})^{12} - (\frac{c}{|x|})^6 \big), \]
where $\varepsilon$ and $c$ are positive constants.
It holds 
\begin{align*}
\nabla_i \ln \phi(x)&=- \sum_{j \neq i} \nabla_i \zeta(x^i -x^j) \\
&= \frac{24\varepsilon}{c^2}~\sum_{j \neq i} \Big(2~\big(\frac{c}{|x^i-x^j|}\big)^{14} - \big(\frac{c}{|x^i-x^j|}\big)^8 \Big) \big(x^i -x^j \big).
\end{align*}
With $f(r):=\frac{24\varepsilon}{c^2} \Big(2~\big(\frac{c}{r}\big)^{14} - \big(\frac{c}{r}\big)^8 \Big)$ we get
\[ \nabla_i \ln \phi(x)=\sum_{j \neq i} f(|x^i-x^j|) \big(x^i -x^j). \]
Thus, the absolute value of the acting force obviously depends only on the distance of the respective particles. In this case, for $\delta=0$ the corresponding system of SDEs is given by
\begin{align*}
d\mathbf{X}^i_t =& \mathbbm{1}_{\Omega} (\mathbf{X}^i_t) \Big( dB^i_t + \frac{1}{2} \frac{\nabla_i \alpha_i}{\alpha_i} (\mathbf{X}^i_t)~ dt+ \frac{1}{2} \sum_{j \neq i} f(|\mathbf{X}_t^i-\mathbf{X}_t^j|) \big(\mathbf{X}_t^i -\mathbf{X}_t^j)~dt \Big)\\
&- \mathbbm{1}_{\Gamma}(\mathbf{X}^i_t) \frac{\alpha_i}{\beta_i}(\mathbf{X}^i_t) ~n(\mathbf{X}^i_t)~ dt, \quad i=1,\dots,N,   \\ 
\mathbf{X}_0 =& x, 
\end{align*}
where $(B_t)_{t \geq 0}$, $B_t=(B_t^1,\dots,B_t^N)$, is an $Nd$-dimensional standard Brownian motion. It is natural that we obtain in this case only a solution for quasi every starting point, since points in $\Lambda$ which describe configurations where two or more particles are at the same position in $\overline{\Omega}$ are naturally not admissible in view of the singularity of $\zeta$ in $0$. An appropriate regularity results regarding the elliptic PDE associated to the form $(\mathcal{E},D(\mathcal{E}))$ would allow to apply the results of \cite{BGS13}. In this case, a process on $\Lambda \backslash \{ \phi=0\}=\{ x=(x^1,\dots,x^N) \in \Lambda|~ x^i \neq x^j \text{ for every } i \neq j \}$ can be constructed which is a solution to the above SDE for every starting point in $\Lambda \backslash \{ \phi=0\}$.
\end{example}

\bibliographystyle{alpha}
\bibliography{biblio}

\end{document}